\documentclass[11pt,twoside]{article}
\usepackage[utf8]{inputenc}
\usepackage{amsmath}
\usepackage{amsfonts}
\usepackage{amssymb}
\usepackage{amsthm}
\usepackage[english]{babel}
\usepackage{hyperref}
\usepackage[alphabetic]{amsrefs}
\usepackage{geometry}
\usepackage{graphicx}
\usepackage{mathrsfs}
\usepackage{parskip}
\usepackage{paralist}

\hypersetup{
	colorlinks   = true, 
	urlcolor     = blue, 
	linkcolor    = blue,
	citecolor    = blue 
}

\usepackage{fancyhdr}
\usepackage{mathtools}
\mathtoolsset{showonlyrefs}

\makeatletter
\def\thm@space@setup{%
	\thm@preskip=0pt \thm@postskip=0pt
}
\makeatother
\usepackage{etoolbox}
\AtBeginEnvironment{proof}{\vspace*{-0.5\parskip}}

\begin{document}
\title{\vspace{-1ex}\textbf{The Allen-Cahn equation on the complete Riemannian manifolds of finite volume}}
\author{Akashdeep Dey \thanks{Email: adey@math.princeton.edu, dey.akash01@gmail.com}
}
\date{}
\maketitle

\theoremstyle{plain}
\newtheorem{thm}{Theorem}[section]
\newtheorem{lem}[thm]{Lemma}
\newtheorem{pro}[thm]{Proposition}
\newtheorem{clm}[thm]{Claim}
\newtheorem*{thm*}{Theorem}
\newtheorem*{lem*}{Lemma}
\newtheorem*{clm*}{Claim}

\theoremstyle{definition}
\newtheorem{defn}[thm]{Definition}
\newtheorem{ex}[thm]{Example}
\newtheorem{rmk}[thm]{Remark}

\numberwithin{equation}{section}

\newcommand{\mf}{manifold\;}\newcommand{\mfs}{manifolds\;}
\newcommand{\vf}{varifold\;}
\newcommand{\hy}{hypersurface\;}
\newcommand{\Rm}{Riemannian\;}
\newcommand{\cn}{constant\;}
\newcommand{\mt}{metric\;} 
\newcommand{\st}{such that\;}
\newcommand{\Thm}{Theorem\;}
\newcommand{\Lem}{Lemma\;}
\newcommand{\Pro}{Proposition\;}
\newcommand{\eqn}{equation\;}
\newcommand{\te}{there exist\;}\newcommand{\tes}{there exists\;}\newcommand{\Te}{There exist\;}\newcommand{\Tes}{There exists\;}\newcommand{\fa}{for all\;}
\newcommand{\tf}{Therefore,\;} \newcommand{\hn}{Hence\;}\newcommand{\Sn}{Since\;}\newcommand{\sn}{since\;}\newcommand{\nx}{\Next,\;}\newcommand{\df}{define\;}
\newcommand{\wrt}{with respect to\;}
\newcommand{\bbr}{\mathbb{R}}\newcommand{\bbq}{\mathbb{Q}}
\newcommand{\bbn}{\mathbb{N}}
\newcommand{\bbz}{\mathbb{Z}}
\newcommand{\mres}{\scalebox{1.8}{$\llcorner$}}
\newcommand{\ra}{\rightarrow}
\newcommand{\fn}{function\;}
\newcommand{\lra}{\longrightarrow}
\newcommand{\sps}{Suppose\;}
\newcommand{\del}{\partial}
\newcommand{\seq}{sequence\;}\newcommand{\w}{with\;}
\newcommand{\cts}{continuous\;} 
\newcommand{\bF}{\mathbf{F}} 
\newcommand{\bM}{\mathbf{M}} 
\newcommand{\bL}{\mathbf{L}}
\newcommand{\cm}{\mathcal{C}(M)}
\newcommand{\zn}{\mathcal{Z}_n(M, \mathbb{Z}_2)}
\newcommand{\qte}{\hat{q}_{\ve}}\newcommand{\qe}{q_{\ve}}\newcommand{\inn}{\mathring{N}}

\newcommand{\cA}{\mathcal{A}}\newcommand{\cB}{\mathcal{B}}\newcommand{\cC}{\mathcal{C}}\newcommand{\cD}{\mathcal{D}}\newcommand{\cE}{\mathcal{E}}\newcommand{\cF}{\mathcal{F}}\newcommand{\cG}{\mathcal{G}}\newcommand{\cH}{\mathcal{H}}\newcommand{\cI}{\mathcal{I}}\newcommand{\cJ}{\mathcal{J}}\newcommand{\cK}{\mathcal{K}}\newcommand{\cL}{\mathcal{L}}\newcommand{\cM}{\mathcal{M}}\newcommand{\cN}{\mathcal{N}}\newcommand{\cO}{\mathcal{O}}\newcommand{\cP}{\mathcal{P}}\newcommand{\cQ}{\mathcal{Q}}\newcommand{\cR}{\mathcal{R}}\newcommand{\cS}{\mathcal{S}}\newcommand{\cT}{\mathcal{T}}\newcommand{\cU}{\mathcal{U}}\newcommand{\cV}{\mathcal{V}}\newcommand{\cW}{\mathcal{W}}\newcommand{\cX}{\mathcal{X}}\newcommand{\cY}{\mathcal{Y}}\newcommand{\cZ}{\mathcal{Z}}

\newcommand{\sA}{\mathscr{A}}\newcommand{\sB}{\mathscr{B}}\newcommand{\sC}{\mathscr{C}}\newcommand{\sD}{\mathscr{D}}\newcommand{\sE}{\mathscr{E}}\newcommand{\sF}{\mathscr{F}}\newcommand{\sG}{\mathscr{G}}\newcommand{\sH}{\mathscr{H}}\newcommand{\sI}{\mathscr{I}}\newcommand{\sJ}{\mathscr{J}}\newcommand{\sK}{\mathscr{K}}\newcommand{\sL}{\mathscr{L}}\newcommand{\sM}{\mathscr{M}}\newcommand{\sN}{\mathscr{N}}\newcommand{\sO}{\mathscr{O}}\newcommand{\sP}{\mathscr{P}}\newcommand{\sQ}{\mathscr{Q}}\newcommand{\sR}{\mathscr{R}}\newcommand{\sS}{\mathscr{S}}\newcommand{\sT}{\mathscr{T}}\newcommand{\sU}{\mathscr{U}}\newcommand{\sV}{\mathscr{V}}\newcommand{\sW}{\mathscr{W}}\newcommand{\sX}{\mathscr{X}}\newcommand{\sY}{\mathscr{Y}}\newcommand{\sZ}{\mathcal{Z}}

\newcommand{\al}{\alpha}\newcommand{\be}{\beta}\newcommand{\ga}{\gamma}\newcommand{\de}{\delta}\newcommand{\ve}{\varepsilon}\newcommand{\et}{\eta}\newcommand{\ph}{\phi}\newcommand{\vp}{\varphi}\newcommand{\ps}{\psi}\newcommand{\ka}{\kappa}\newcommand{\la}{\lambda}\newcommand{\om}{\omega}\newcommand{\rh}{\rho}\newcommand{\si}{\sigma}\newcommand{\tht}{\theta}\newcommand{\ta}{\tau}\newcommand{\ch}{\chi}\newcommand{\ze}{\zeta}\newcommand{\Ga}{\Gamma}\newcommand{\De}{\Delta}\newcommand{\Ph}{\Phi}\newcommand{\Ps}{\Psi}\newcommand{\La}{\Lambda}\newcommand{\Om}{\Omega}\newcommand{\Si}{\Sigma}\newcommand{\Tht}{\Theta}\newcommand{\na}{\nabla}
\newcommand{\ep}{\epsilon}\newcommand{\vt}{\vartheta}

\newcommand{\fu}{\mathfrak{u}}\newcommand{\HN}{H^1(N)}\newcommand{\phh}{\hat{\phi}}\newcommand{\gab}{\bar{\gamma}}
\newcommand{\nm}[1]{\left\|#1\right\|}\newcommand{\md}[1]{\left|#1\right|}\newcommand{\Md}[1]{\Big|#1\Big|}\newcommand{\db}[1]{[\![#1]\!]}
\newcommand{\vol}{\operatorname{Vol}}\newcommand{\Ee}[1]{E_{\varepsilon}(#1)}\newcommand{\Eee}[1]{E_{\varepsilon}\left(#1\right)}\newcommand{\Ea}[1]{E_{\alpha}(#1)}\newcommand{\Eaa}[1]{E_{\alpha}\left(#1\right)}\newcommand{\Eab}[1]{E_{\alpha}\big(#1\big)}
\newcommand{\se}{\sqrt{\ve}}\newcommand{\woe}{w_{1,\ve}(t)}\newcommand{\wte}{w_{2,\ve}(t)}\newcommand{\we}{w_{\al}}\newcommand{\htl}{\tilde{h}}\newcommand{\htt}{\tilde{h}(t^*)}\newcommand{\lte}{\tilde{\la}_{\ve}}\newcommand{\bi}{b_{i}^{r_i}}\newcommand{\bk}{b_{k}^{r_k}}
\vspace{-2ex}
\begin{abstract}
	\vspace{-1.5ex}
	\noindent
The semi-linear, elliptic PDE \(AC_{\ve}(u):=-\ve^2\De u+W'(u)=0\) is called the Allen-Cahn equation. In this article we will prove the existence of finite energy solution to the Allen-Cahn equation on certain complete, non-compact manifolds. More precisely, suppose \(M^{n+1}\) (with \(n+1\geq 3\)) is a complete Riemannian \mf of finite volume. Then \tes \(\ve_0>0\), depending on the ambient \Rm metric, \st \fa \(0<\ve\leq\ve_0\), \tes\(\fu_{\ve}:M\ra (-1,1)\) satisfying \(AC_{\ve}(\fu_{\ve})=0\) with the energy \(E_{\ve}(\fu_{\ve})<\infty\) and the Morse index \(\textup{Ind}(\fu_{\ve})\leq 1\).  Moreover, \(0<\liminf_{\ve\ra 0}E_{\ve}(\fu_{\ve})\leq\limsup_{\ve\ra 0}E_{\ve}(\fu_{\ve})<\infty.\) Our result is motivated by the theorem of Chambers-Liokumovich \cite{CL} and Song \cite{Song2}, which says that \(M\) contains a complete minimal \hy \(\Si\) \w \(0<\cH^n(\Si)<\infty.\) This theorem can be recovered from our result.
\end{abstract}

\section{Introduction}
Minimal hypersurfaces are the critical points of the area functional.  By the combined works of Almgren \cite{Alm}, Pitts \cite{Pit} and Schoen-Simon \cite{SS}, every closed \Rm \mf \\
\((M^{n+1},g)\), \(n+1 \geq 3\), contains a closed minimal hypersurface, which is smooth and embedded outside a singular set of Hausdorff dimension \(\leq n-7\). 

Recently, Almgren-Pitts min-max theory has been further extended and it has been discovered that minimal hypersurfaces exist in abundance. By the works of Marques-Neves \cite{MN_ricci_positive} and Song \cite{Song1}, every closed Riemannian manifold \((M^{n+1},g)\), \(3\leq n+1\leq 7\), contains infinitely many closed, minimal hypersurfaces. This was conjectured by Yau \cite{Yau}. In \cite{IMN}, Irie, Marques and Neves proved that for a generic \mt \(g\) on \(M\), the union of all closed, minimal hypersurfaces is dense in \((M,g)\). This theorem was later quantified by Marques, Neves and Song in \cite{MNS} where they proved that for a generic metric there exists an equidistributed sequence of closed, minimal hypersurfaces in \((M,g)\). Recently, Song and Zhou \cite{SZ} proved the generic scarring phenomena for minimal hypersurfaces, which can be interpreted as the opposite of the equidistribution phenomena. In \cite{Zhou}, Zhou proved that for a generic (bumpy) metric, the min-max minimal hypersurfaces have multiplicity one, which was conjectured by Marques and Neves. Using this theorem, Marques and Neves \cite{MN_index, MN_index_2} proved that for a generic (bumpy) \mt \(g\), there exists a sequence of closed, two sided minimal hypersurfaces \(\{\Si_k\}_{k=1}^{\infty}\) in \((M,g)\) \st Ind\((\Si_k)=k\) and \(\mathcal{H}^n(\Si_k) \sim k^{\frac{1}{n+1}}\). In higher dimensions, Li \cite{Li} proved the existence of infinitely many closed minimal hypersurfaces (with optimal regularity) for a generic set of metrics. While the arguments in \cite{IMN}, \cite{MNS} and \cite{Li} depend on the Weyl law for the volume spectrum, which was conjectured by Gromov \cite{Gro1} and proved by Liokumovich, Marques and Neves \cite{LMN}, the arguments in \cite{Song1} and \cite{SZ} use the cylindrical Weyl law, which was proved by Song \cite{Song1}.

In the above mentioned theorems, the ambient manifolds are assumed to be closed. If \(M\) is a complete non-compact manifold, Gromov \cite{Gro2} proved that either \(M\) contains a complete minimal \hy with finite area or every compact domain of \(M\) admits a (possibly singular) strictly mean convex foliation. In \cite{Mon}, Montezuma proved that a complete \Rm \mf with a bounded, strictly mean concave domain contains a complete minimal hypersurface with finite area. The existence of minimal surfaces in hyperbolic $3$-manifolds has been proved by Collin-Hauswirth-Mazet-Rosenberg \cite{CHMR}, Huang-Wang \cite{HW} and Coskunuzer \cite{C}. In \cite{CK}, Chodosh and Ketover proved the existence of minimal planes in asymptotically flat \(3\)-manifolds. In \cite{CL}, Chambers and Liokumovich proved that every complete \Rm \mf with finite volume contains a complete minimal hypersurface with finite area. In \cite{Song2}, Song proved Yau's conjecture on certain complete non-compact manifolds. Moreover, he also proved the local version of the above mentioned theorem of Gromov \cite{Gro2}, using which he gave alternative proofs of the above mentioned theorems of Montezuma \cite{Mon} and Chambers-Liokumovich \cite{CL}.

In \cite{G}, Guaraco introduced a new approach for the min-max construction of minimal hypersurfaces, which was further developed by Gaspar and Guaraco in \cite{GG1}. This approach is based on the study of the limiting behaviour of solutions to the Allen-Cahn equation. The Allen-Cahn equation (with parameter \(\ve >0\)) is the following semi-linear, elliptic PDE
\begin{equation}\label{AC.eqn}
AC_{\ve}(u):=-\ve^{2}\De u+ W'(u)=0
\end{equation}
where \(W:\bbr \ra [0,\infty)\) is a double well potential e.g. \(W(t)=\frac{1}{4}(1-t^2)^2.\) The solutions of this equation are precisely the critical points of the energy functional
\[E_{\ve}(u)=\int_M\ve\frac{|\nabla u|^2}{2}+\frac{W(u)}{\ve}.\]
Informally speaking, as \(\ve\ra 0\), the level sets of the solutions to \eqref{AC.eqn} (with uniformly bounded energy) accumulate around a generalized minimal hypersurface (called a \textit{limit-interface}). In particular, Modica \cite{M} and Sternberg \cite{S} proved that as \(\ve\ra 0\), the energy minimizing solutions to \eqref{AC.eqn} converge to a area minimizing hypersurface. For general solutions to \eqref{AC.eqn}, Hutchinson and Tonegawa \cite{HT} proved that the limit-interface is a stationary, integral varifold. Moreover, if the solutions are stable, by the works of Tonegawa \cite{Ton}, Wickramasekera \cite{W} and Tonegawa-Wickramasekera \cite{TW}, the limit-interface is a stable minimal \hy with optimal regularity. In \cite{G}, Guaraco proved that the limit-interface has optimal regularity if the solutions have uniformly bounded Morse index. Furthermore, by a mountain-pass argument, he proved the existence of critical points of \(E_{\ve}\) (on a closed \Rm manifold) with uniformly bounded energy and Morse index. In this way he obtained a new proof of the previously mentioned theorem of Almgren-Pitts-Schoen-Simon. In the case of surfaces (i.e. when the ambient dimension $=2$), Mantoulidis \cite{Man1} proved the regularity of the geodesic limit-interface for the solutions with uniformly bounded Morse index. 

The index upper bound of the limit-interface was proved by Hiesmayr \cite{H} assuming the limit-interface is two-sided and by Gaspar \cite{Gaspar} in the general case. In \cite{GG2}, Gaspar and Guaraco proved the Weyl law for the phase transition spectrum and gave alternative proofs of the density \cite{IMN} and the equidistribution \cite{MNS} theorems. In \cite{CM}, Chodosh and Mantoulidis proved the multiplicity one conjecture in the Allen-Cahn setting in dimension \(3\) and the upper semi-continuity of the Morse index when the limit-interface has multiplicity one. As a consequence, they proved that for a generic (bumpy) metric \(g\) on a closed manifold \(M^3\), there exists a sequence of closed, two-sided minimal surfaces \(\{\Si_p\}_{p=1}^{\infty}\) in \((M^3,g)\) \st \(\text{Ind}(\Si_p)=p\) and \(\text{area}(\Si_p)\sim p^{1/3}\). In higher dimensions, the multiplicity one conjecture for the one parameter Allen-Cahn min-max has been proved by Bellettini \cite{B1,B2}. In \cite{GMN} Guaraco, Marques and Neves proved that a strictly stable limit-interface must have multiplicity one. 

In \cite{BW3}, Bellettini and Wickramasekera proved the existence of closed prescribed mean curvature (PMC) hypersurfaces in arbitrary closed \Rm \mfs using the min-max solutions of the inhomogeneous Allen-Cahn equations. To prove the regularity of the Allen-Cahn PMC hypersurfaces, they used their earlier works \cite{BW1,BW2} on the regularity and compactness theory of stable PMC hypersurfaces. Previously, Zhou and Zhu \cite{ZZ1,ZZ2} developed a min-max theory for the construction of closed PMC hypersurfaces which is parallel to the Almgren-Pitts min-max theory. The estimates for the index and nullity of the Allen-Cahn PMC hypersurfaces have been proved by Mantoulidis \cite{Man2}. 

The asymptotic behaviour of the critical points of the Ginzburg–Landau functional (which approximates the codimension-$2$ area functional) has been studied by Stern \cite{Stern1,Stern2}, Cheng \cite{Cheng} and Pigati-Stern \cite{PS}. In particular, in \cite{PS} Pigati and Stern proved the existence of a codimension-$2$ stationary, integral varifold in an arbitrary closed \Rm manifold. This theorem was previously proved by Almgren \cite{Alm} using more complicated geometric measure theory approach. 

If \(\Si\) is a non-degenerate, separating, closed minimal hypersurface in a closed Riemannian manifold, Pacard and Ritor\'{e} \cite{PR} constructed solutions of the Allen-Cahn equation, for sufficiently small \(\ve>0\), whose level sets converge to \(\Si\). The uniqueness of these solutions has been proved by Guaraco, Marques and Neves \cite{GMN}. The construction of Pacard and Ritor\'{e} has been extended by Caju and Gaspar \cite{CG} in the case when all the Jacobi fields of \(\Si\) are induced by the ambient isometries. Assuming a positivity condition on the Ricci curvature of the ambient manifold, del Pino-Kowalczyk-Wei-Yang \cite{dPKWY} constructed solutions of the Allen-Cahn equation whose energies concentrate on a non-degenerate, closed minimal hypersurface with multiplicity \(>1\).

In this article we will show the existence of finite energy min-max solution to the Allen-Cahn equation (for $\ve$ sufficiently small) on complete \Rm manifolds of finite volume. More precisely, we will prove the following theorem, which is motivated by the previously mentioned theorem of Chambers-Liokumovich \cite{CL} and Song \cite{Song2}.

\begin{thm}\label{t:main.thm}
Let \(M^{n+1}\) be a complete, \Rm manifold, \(n+1\geq 3\), \st \(\vol(M)\) is finite. Then there exists \(\ve_0>0\), depending on the ambient \Rm metric, \st for all \(0<\ve\leq \ve_0\), \tes \(\fu_{\ve}:M\ra (-1,1)\) satisfying \(AC_{\ve}(\fu_{\ve})=0\), \(\textup{Ind}(\fu_{\ve})\leq 1\) and \(E_{\ve}(\fu_{\ve})<\infty\). Moreover, \tes a good set \(U\subset M\) (see Section \ref{s.2.2}) \st 
\begin{equation}\label{e1}
0<\liminf_{\ve\ra 0^{+}}E_{\ve}(\fu_{\ve},U)\leq\limsup_{\ve\ra 0^+}E_{\ve}(\fu_{\ve})<\infty.
\end{equation}
\end{thm}

We will prove Theorem \ref{t:main.thm} by adapting the argument of Chambers-Liokumovich \cite{CL} in the Allen-Cahn setting. From Theorem \ref{t:main.thm}, one can recover the above mentioned theorem of Chambers-Liokumovich \cite{CL} and Song \cite{Song2}.

\begin{thm}\cite{CL,Song2}\label{t.minimal.hyp}
Let \(M^{n+1}\) be a complete, \Rm manifold, \(n+1\geq 3\), \st \(\vol(M)\) is finite. Then there exists a complete minimal \hy \(\Si\subset M\) \st \(0<\cH^n(\Si)<\infty\) and \(\Si\) has optimal regularity, i.e. \(\Si\) is smooth and embedded outside a singular set of Hausdorff dimension \(\leq n-7.\)
\end{thm}

As in \cite{CL} and \cite{Song2}, Theorem \ref{t:main.thm} and Theorem \ref{t.minimal.hyp} continue to hold if the assumption \(\vol(M)<\infty\) is replaced by the weaker assumption that  \tes a sequence \(\{U_i\}_{i=1}^{\infty}\), where each \(U_i\subset M\) is a bounded open set \w smooth boundary, \st \(U_{i}\subset U_{i+1}\) for all \(i\in \bbn\) and $\lim\limits_{i\ra \infty}\cH^n(\del U_i)=0$.

\textbf{Acknowledgements.}  I am very grateful to my advisor Prof. Fernando Cod\'{a} Marques for many helpful discussions and for his support and guidance. The author is partially supported by NSF grant DMS-1811840.

\section{Notation and Preliminaries}
\subsection{Notation}
Here we summarize the notation which will be frequently used later.
\begin{itemize}
	\item \(\cH^k\) : the Hausdorff measure of dimension \(k\).
	\item \(B(p,r)\) : the geodesic ball centered at \(p\) with radius \(r\).
	\item \(d(-,S)\) : distance from a set \(S\).
	\item \(H^1(N)\) : the Sobolev space \(\left\{u\in L^2(N):\text{ the distributional derivative }\na u \in L^2(N,TN)\right\}\).
	\item \(e_{\ve}(u)\)  \(=\ve \frac{|\na u|^2}{2}+\frac{W(u)}{\ve}\).
	\item 	\(E_{\ve}(u)\)  \(=\int_N e_{\ve}(u)\), where \(N\) is the ambient manifold.
	\item \(E_{\ve}(u,S)\)  \(=\int_{S} e_{\ve}(u)\), where \(S\) is a measurable set.
	\item \(AC_{\ve}(u)\)  \(=-\ve^2\De u+W'(u)\).
	\item \(2\si\) \(=\) the energy of the \(1\)-D solution to the Allen-Cahn equation (see \eqref{e.def.F}, \eqref{e.energy.1D.soln}).
	\item For two measurable functions \(u\) and \(v\), we say that \(u\leq v\) (resp. \(u\geq v\)) if \(u(x)\leq v(x)\) (resp. \(u(x)\geq v(x)\)) for a.e. \(x\).
\end{itemize}
\subsection{The Allen-Cahn equation and convergence of the phase interfaces}
In this subsection we will briefly discuss about the Allen-Cahn equation and its connection with the minimal hypersurfaces. \sps $\Om^{n+1}$ is the interior of a compact \Rm manifold. Let \(W:\bbr \ra [0,\infty)\) be a smooth, symmetric, double well potential. More precisely, \(W\) has the following properties. $W$ is bounded; \(W(-t)=W(t)\) for all \(t \in \bbr\); \(W\) has exactly three critical points \(0,\pm 1\); \(W(\pm 1)=0\) and \(W''(\pm 1)>0\) i.e. \(\pm 1\) are non-degenerate minima; \(0\) is a local maximum. For $u \in H^1(\Om)$, the \textit{$\ve$-Allen-Cahn energy of $u$} is given by 
\[E_{\ve}(u)=\int_{\Om}\ve\frac{|\nabla u|^2}{2}+\frac{W(u)}{\ve}.\]
As mentioned earlier, 
\[AC_{\ve}(u):=-\ve^2\De u+ W'(u)=0\]
if and only if \(u\) is a critical point of \(E_{\ve}\).

Let \(F:\bbr \ra \bbr\) and the energy constant \(\si\) be defined as follows.
\begin{equation}\label{e.def.F}
F(t)=\int_{0}^{t}\sqrt{W(s)/2}\; ds;\quad \quad \si=\int_{-1}^{1}\sqrt{W(s)/2}\;ds\quad \text{so that}\quad F(\pm 1)=\pm \frac{\si}{2}.
\end{equation}
For an \(n\)-rectifiable set \(S\subset \Om\), let $\md{S}$ denote the $n$-varifold defined by $S$. Given \(u \in C^1(\Om)\), we set \(\tilde{u}=F \circ u\). The \(n\)-varifold associated to \(u\) is defined by
\[V[u](A)=\frac{1}{\si}\int_{-\infty}^{\infty}\md{\{\tilde{u}=s\}}(A)\; ds,\]
for every Borel set \(A \subset G_n\Om\) (where $G_n\Om$ denotes the Grassmannian bundle of unoriented $n$-dimensional hyperplanes on $\Om$).

Building on the works of Hutchinson-Tonegawa \cite{HT}, Tonegawa \cite{Ton} and Tonegawa-Wickramasekera \cite{TW}, Guaraco \cite{G} has proved the following theorem.

\begin{thm}[\cite{HT,Ton,TW,G}]\label{thm interface}
Suppose \(\Om^{n+1}\), \(n+1\geq 3\), is the interior of a compact \Rm manifold. Let \(\{u_{i}:\Om\ra (-1,1)\}_{i=1}^{\infty}\) be a sequence of smooth functions \st
	\begin{itemize}
		\item[(i)] \(AC_{\ep_i}(u_i)=0\) with \(\ep_i\ra 0\) as \(i \ra \infty\);
		\item[(ii)] \[\sup_{i\in \bbn}\;E_{\ep_i}(u_i)< \infty\quad \text{ and }\quad \sup_{i\in \bbn}\;\textup{Ind}(u_i)<\infty.\]
	\end{itemize}
Then there exists a stationary, integral varifold \(V\) in \(\Om\) \st possibly after passing to a subsequence, \(V[u_i]\ra V\) in the sense of varifolds. Moreover, \(\textup{spt}(V)\) is a minimal hypersurface with optimal regularity in \(\Om\). Furthermore, if \(\nm{V}\) denotes the Radon measure associated to \(V\), then
\begin{equation}\label{e.rad.measure.V}
\frac{1}{2\si}\left(\ep_i\frac{\md{\na u_i}^2}{2}+\frac{W(u_i)}{\ep_i}\right)d\textup{Vol}_{\Om}\ra \nm{V},
\end{equation}
in the sense of Radon measures.
\end{thm}
The proof of the regularity of the limit-interface depends on the regularity theory of stable, minimal hypersurfaces, developed by Wickramasekera \cite{W}. In the ambient dimension \(n+1=3\), the regularity of the limit-interface can also be obtained from the curvature estimates of Chodosh and Mantoulidis \cite{CM}.  

We also state here the theorem proved by Smith \cite{Smith} about the generic finiteness of the number of solutions to the Allen-Cahn equation on a closed manifold. This theorem will be used to prove the Morse index upper bound in Theorem \ref{t:main.thm}.

Let \(N^{n+1}\), \(n+1\geq 3\), be a closed manifold and \(\cM\) be the space of all smooth \Rm metrics on \(N\), endowed with the \(C^{\infty}\) topology. For \(\ve>0\) and \(\ga\in \cM\), we define
\begin{equation}
\cZ_{\ve,\ga}=\left\{u\in C^{\infty}(N):-\ve^2\De_{\ga}u+W'(u)=0\right\}.
\end{equation}

\begin{thm}\cite{Smith}*{Theorem 1.1 (2)}\label{t.finite}
There exists a generic set \(\widetilde{\cM}\subset \cM\) \st if \(\ga\in \widetilde{\cM}\) and \(\ve^{-1}\notin\textup{Spec}(-\De_{\ga})\) (here we are using the convention that \(\textup{Spec}(-\De_{\ga})\subset [0,\infty)\)), then \(\cZ_{\ve,\ga}\) is finite.
\end{thm}

\subsection{Min-max theorem on the Hilbert space}
Let $\sH$ be a separable Hilbert space and \(\cE:\sH\ra \bbr\) be a \(C^2\) functional. \sps \(B_0,B_1\) are closed subsets of \(\sH\). We define
\begin{equation}\label{def.sF}
\sF=\left\{\ze:[0,1]\ra \sH: \ze\text{ is continuous, } \ze(0)\in B_0,\;\ze(1)\in B_1 \right\}
\end{equation}
and 
\begin{equation}
c=\inf_{\ze\in\sF}\sup_{t\in [0,1]}\cE(\ze(t)).
\end{equation}
A sequence \(\{\ze_i\}_{i=1}^{\infty}\subset \sF\) is called a \textit{minimizing \seq} if
\[\lim_{i\ra\infty}\sup_{t\in [0,1]} \cE(\ze_i(t))=c.\]
For a minimizing sequence \(\{\ze_i\}\subset \sF\), let \(\cK\left(\{\ze_i\}\right)\) denote the set of all \(v\in \sH\) for which \te sequences \(\{i_j\}\subset \{i\}\) and \(\{t_i\}\subset [0,1]\) \st
\[v=\lim_{j\ra \infty}\ze_{i_j}(t_j).\]
\begin{defn} 
Given a minimizing \seq \(\{\ze_i\}\) in \(\sF\), we say that \(\cE\) satisfies the \textit{Palais-Smale condition} along \(\{\ze_i\}\) if every sequence \(\{v_i\}\), satisfying the conditions
\[\lim_{i\ra\infty}\cE'(v_i)=0\quad \text{ and }\quad \lim_{i\ra\infty}d\left(v_i,\ze_i\left([0,1]\right)\right)=0,\]
has a convergent subsequence.
\end{defn}

\begin{defn}\cite{Gho}*{Section 3, page 53}
Let
\[K_c=\left\{v\in \sH:\cE'(v)=0,\;\cE(v)=c\right\}.\]
A compact subset \(\sC \text{ of } K_c\) is called an \textit{isolated critical set} for \(\cE\) in \(K_c\) if \tes an open set \(\sU\subset \sH\) \st \(\sC\subset \sU\) and
\[K_c\cap \sU = \sC.\]
\end{defn}

The following min-max theorem, which was proved by Ghoussoub \cite{Gho} in a much more general setting, will be used to prove Theorem \ref{t:main.thm}.

\begin{thm}\cite{Gho}\label{t.min.max}
(a) Let \(L\subset \sH\) be a closed set \st the following conditions are satisfied:
\begin{itemize}
	\item[(a1)] \(L\cap (B_0\cup B_1)=\emptyset\);
	\item[(a2)] for all \(\ze\in \sF\), \(L\cap \ze\left([0,1]\right)\neq \emptyset\);
	\item[(a3)]\(\inf\limits_{v\in L}\cE(v)\geq c.\)
\end{itemize}
Suppose \(\cE\) satisfies the Palais-Smale condition along a minimizing \seq \(\{\ze_i\}_{i=1}^{\infty}\). Then 
\begin{equation}
K_c\cap L \cap \cK\left(\{\ze_i\}\right)\neq \emptyset.
\end{equation}

(b) In addition to the assumptions stated in part (a), let us also assume that:
\begin{itemize}
	\item[(b1)] \(K_c\cap L\) is an isolated critical set for \(\cE\) in \(K_c\);
	\item[(b2)] \(\cE''\) is Fredholm on \(K_c\).
\end{itemize}
Then there exists
\begin{equation}
v\in K_c\cap L \cap \cK\left(\{\ze_i\}\right) \text{ such that } m(v)\leq 1,
\end{equation}
where \(m(v)\) is the Morse index of the critical point \(v\), i.e. \(m(v)\) is equal to the index of the bilinear form \(\cE''\big|_v\).
\end{thm}

\begin{rmk}
In the definition of \(\sF\) in \eqref{def.sF} and in Theorem \ref{t.min.max}, we have assumed that \(B_0\) and \(B_1\) are closed subsets of \(\sH\). This is slightly different from the hypothesis made in \cite{Gho}*{Theorem (1.bis) and Theorem (4)}, where \(B_0\) and \(B_1\) are assumed to be singleton sets. However this does not affect the proof of Theorem \ref{t.min.max} in \cite{Gho} for the following reason (see \cite{Gho}*{Remark (3) in page 32 and Remark (11) in page 60}). If \(\ze\in \sF\) (as defined in \eqref{def.sF}) and \(\ze':[0,1]\ra \sH\) is another map satisfying \(\ze'(0)=\ze(0)\) and \(\ze'(1)=\ze(1)\), then \(\ze'\in \sF\) as well.
\end{rmk}

\subsection{The notion of the good set}\label{s.2.2}
In this subsection we will recall the definition of the \textit{good set} from \cite{CL}*{Section 2.2}. Let \(N\) be a complete \Rm \mf and \(\Om\subset N\) be a bounded open set with smooth boundary \(\del\Om\). \(\mathbf{I}_{n+1}(\Om;\bbz_2)\) denotes the space of \((n+1)\)-dimensional mod \(2\) flat chains in \(\Om\); \(\cZ_{n,\text{rel}}(\Om,\del\Om;\bbz_2)\) denotes the space of \(n\)-dimensional mod \(2\) relative flat cycles in \(\Om\) and \(\del:\mathbf{I}_{n+1}(\Om;\bbz_2)\ra \cZ_{n,\text{rel}}(\Om,\del\Om;\bbz_2)\) is the boundary map. Both the spaces \(\mathbf{I}_{n+1}(\Om;\bbz_2)\) and \(\cZ_{n,\text{rel}}(\Om,\del\Om;\bbz_2)\) are assumed to be equipped with the flat topology. (We refer to \cite{LMN}*{Section 2} and \cite{LZ}*{Section 3} for more details about these spaces.) Let \(\sS\) be the set of all \cts maps \(\Ga:[0,1]\ra \mathbf{I}_{n+1}(\Om;\bbz_2)\) \st \(\Ga(0)=\emptyset\) and \(\Ga(1)=\Om\). The \textit{(relative) width of \(\Om\)}, denoted by \(\mathbb{W}(\Om)\), is defined as follows \cite{Gro, Guth, LMN, CL, LZ}.
\begin{equation}\label{e.def.rel.width}
\mathbb{W}(\Om)=\inf_{\Ga\in \sS}\sup_{t\in [0,1]}\mathbf{M}(\del\Ga(t)).
\end{equation}
$\Om$ is called a \textit{good set} if
\begin{equation}\label{good.set}
\mathbb{W}(\Om)> 4\cH^n(\del \Om).
\end{equation}
\section{Nested families in the Sobolev space \(H^1(N)\)}
The notion of the nested family of open sets played an important role in the proof of the main theorem of Chambers-Liokumovich in \cite{CL}. In this section, we will deal with the notion of the nested family in the function space \(H^1(N)\). Throughout this section, \(N\) will be assumed to be a closed \Rm manifold (of dimension \(n+1\)) and \(\ve>0\). We begin with the following lemma.
\begin{lem}\label{l.minimization.prob}
Let \(u,v\in H^1(N)\) \st \(v\geq u\) and \(|u|,|v|\leq 1\). \sps 
\[\cS=\{w\in H^1(N):v\geq w\geq u\}.\]
Then \tes \(w^* \in \cS\) \st
\[E_{\ve}(w^*)=\inf\left\{E_{\ve}(w):w\in \cS\right\}.\]
\end{lem}
\begin{proof}
	Let
\[\al=\inf\left\{E_{\ve}(w):w\in \cS\right\}\]
and \(\{w_{i}\}_{i=1}^{\infty}\subset \cS\) be \st 
\begin{equation}\label{e.minimizing.seq}
\int_N\ve\frac{|\na w_i|^2}{2}+\frac{W(w_i)}{\ve}\leq \al+\frac{1}{i}.
\end{equation}
Since \(|u|,|v|\leq 1\),
\begin{equation}\label{e.2}
|w_i|\leq 1\quad\forall\;i\in\bbn.
\end{equation}
\eqref{e.minimizing.seq} and \eqref{e.2} imply that \(\{w_i\}_{i=1}^{\infty}\) is a bounded \seq in \(H^1(N)\). \tf by Rellich's compactness theorem, \te \(w^*\in H^1(N)\) and a sub\seq \(\{w_{i_k}\}_{k=1}^{\infty}\) \st 
\begin{equation}\label{e.3}
w_{i_k}\ra w^*\text{ strongly in }L^2(N)\text{ and pointwise a.e.}
\end{equation}
and
\begin{equation}\label{e.4}
\na w_{i_k} \ra \na w^*\text{ weakly in }L^2(N).
\end{equation}
\eqref{e.3} implies that \(w^*\in\cS\). \eqref{e.3}, \eqref{e.4} and \eqref{e.minimizing.seq} together imply that \(E_{\ve}(w^*)\leq \al\) and hence \(E_{\ve}(w^*)= \al\) (as \(w^*\in\cS\)).
\end{proof}
\begin{defn}
	A \cts map $u:[a,b]\ra H^1(N)$ is called \textit{nested} if \(u(t)\geq u(s)\) whenever \(t\leq s\).
\end{defn}
\subsection{Truncation and concatenation of the nested maps}
The following Lemmas \ref{l.trancation} and \ref{l.concatenation} are the Allen-Cahn counterparts of \cite{CL}*{Lemma 5.1} and \cite{CL}*{Proposition 6.3}, respectively.
\begin{lem}\label{l.trancation}
(a) Let \(u:[0,1]\ra H^1(N)\) be nested. \sps \(v\in H^1(N)\) has the following properties: 
\begin{itemize}
	\item \(v\geq u(1)\);
	\item for any \(v'\in H^1(N)\) with \(v\geq v'\geq u(1)\), we have \(E_{\ve}(v)\leq E_{\ve}(v')\).
\end{itemize}
Then \tes \(\tilde{u}:[0,1]\ra H^1(N)\) \st
        \begin{itemize}
			\item[(i)] \(\tilde{u}\) is nested;
			\item[(ii)] \(\tilde{u}(0)\geq u(0)\) and \(\tilde{u}(1)=v\); moreover if \(u(0)\geq v\), one can choose \(\tilde{u}(0)= u(0)\);
			\item[(iii)] \(E_{\ve}(\tilde{u}(t))\leq E_{\ve}(u(t))\) \fa \(t\in [0,1]\);
			\item[(iv)] if \(\nm{v}_{L^{\infty}(N)}\leq 1\) and \(\sup\limits_{t\in [0,1]}\nm{u(t)}_{L^{\infty}(N)}\leq 1\), then \(\sup\limits_{t\in [0,1]}\nm{\tilde{u}(t)}_{L^{\infty}(N)}\leq 1\) as well.
		\end{itemize}
\medskip
(b) Let \(u:[0,1]\ra H^1(N)\) be nested. \sps \(v\in H^1(N)\) has the following properties:
\begin{itemize}
	\item \(u(0)\geq v\);
	\item for any \(v'\in H^1(N)\) with \(u(0)\geq v'\geq v\), we have \(E_{\ve}(v)\leq E_{\ve}(v')\).
\end{itemize}
Then \tes \(\tilde{u}:[0,1]\ra H^1(N)\) \st
\begin{itemize}
	\item[(i)] \(\tilde{u}\) is nested;
	\item[(ii)] \(\tilde{u}(0)=v\) and \(\tilde{u}(1)\leq u(1)\); moreover if \(u(1)\leq v\), one can choose \(\tilde{u}(1)=u(1)\);
	\item[(iii)] \(E_{\ve}(\tilde{u}(t))\leq E_{\ve}(u(t))\) \fa \(t\in [0,1]\);
	\item[(iv)] if \(\nm{v}_{L^{\infty}(N)}\leq 1\) and \(\sup\limits_{t\in [0,1]}\nm{u(t)}_{L^{\infty}(N)}\leq 1\), then \(\sup\limits_{t\in [0,1]}\nm{\tilde{u}(t)}_{L^{\infty}(N)}\leq 1\) as well.
\end{itemize}
\end{lem}
\begin{proof}
To prove part (a), we define
\begin{equation*}\label{e.def.tilde.u}
\tilde{u}(t)=\max\{u(t),v\}.
\end{equation*}
Items (i) and (ii) follow from the assumptions that \(u\) is nested and \(v\geq u(1)\), respectively. Item (iv) follows from the definition of $\tilde{u}$. To prove (iii), we consider 
\begin{equation*}
u'(t)= \min\{u(t),v\}.
\end{equation*}
Then \(v\geq u'(t)\geq u(1)\) \fa \(t \in [0,1]\). By our hypothesis,
\[E_{\ve}(v)\leq E_{\ve}(u'(t))=E_{\ve}(u(t),\{u(t)<v\})+E_{\ve}(v,\{u(t)\geq v\}).\]
Hence
\begin{equation}\label{e.q1}
\Ee{v,\{u(t)<v\}}\leq \Ee{u(t),\{u(t)<v\}}.
\end{equation}
\tf using \eqref{e.q1},
\begin{equation*}
\Ee{\tilde{u}(t)}=\Ee{u(t),\{u(t)\geq v\}} + \Ee{v, \{u(t)<v\}}\leq \Ee{u(t)}.
\end{equation*}

Part (b) can also be proved in a similar way by defining 
\[\tilde{u}(t)=\min\{u(t),v\}\]
and using the fact that
\[u(0)\geq \max\{u(t),v\}\geq v.\]
\end{proof}

\begin{lem}\label{l.concatenation} Let \(u_1,\;u_2:[0,1]\ra H^1(N)\) be \st
	\begin{itemize}
		\item \(u_1,\;u_2\) are nested;
		\item \(\sup\limits_{t\in [0,1]}\nm{u_i(t)}_{L^{\infty}(N)}\leq 1\) for $i=1,2$;
		\item \(u_2(0)\geq u_1(1)\);
		\item \(\underset{t\in [0,1]}{\sup} E_{\ve}(u_i(t))\leq A\) for \(i=1,2\).
	\end{itemize}
Then \tes \(\tilde{u}:[0,1]\ra H^1(N)\) \st
\begin{itemize}
	\item[(i)] \(\tilde{u}\) is nested;
	\item[(ii)] \(\sup\limits_{t\in [0,1]}\nm{\tilde{u}(t)}_{L^{\infty}(N)}\leq 1\);
	\item[(iii)] \(\tilde{u}(0)\geq u_1(0)\) and \(\tilde{u}(1)\leq u_2(1)\);
	\item[(iv)] \(\sup\limits_{t\in [0,1]}\Ee{\tilde{u}(t)}\leq A\).
\end{itemize}
\end{lem}
\begin{proof}
Let
\[\cS=\{v\in \HN : u_2(0)\geq v\geq u_1(1)\}.\]
By Lemma \ref{l.minimization.prob}, \tes \(v^*\in \cS\) \st
\[E_{\ve}(v^*)=\inf\{\Ee{v}:v\in \cS\}.\]
We note that \(v^*\geq v'\geq u_1(1)\) implies that \(v'\in \cS\) and hence \(\Ee{v^*}\leq \Ee{v'}\). \tf by Lemma \ref{l.trancation}, part (a), \tes a nested map \(\tilde{u}_1:[0,1]\ra \HN\) \st 
\[\tilde{u}_1(0)\geq u_1(0),\;\tilde{u}_1(1)=v^*,\; \sup\limits_{t\in [0,1]}\nm{\tilde{u}_1(t)}_{L^{\infty}(N)}\leq 1\; \text{ and }\; \sup\limits_{t\in [0,1]}\Ee{\tilde{u}_1(t)}\leq A.\]
Similarly, \(u_2(0)\geq v'\geq v^* \) implies that \(v'\in \cS\) and hence \(\Ee{v^*}\leq \Ee{v'}\). \tf by Lemma \ref{l.trancation}, part (b), \tes a nested map \(\tilde{u}_2:[0,1]\ra \HN\) \st 
\[\tilde{u}_2(0)=v^*,\; \tilde{u}_2(1)\leq u_2(1),\; \sup\limits_{t\in [0,1]}\nm{\tilde{u}_2(t)}_{L^{\infty}(N)}\leq 1\; \text{ and } \sup\limits_{t\in [0,1]}\Ee{\tilde{u}_2(t)}\leq A.\]
Finally, we define \(\tilde{u}:[0,1]\ra H^1(N)\) by
\begin{equation*}
\tilde{u}(t)=\begin{cases}
\tilde{u}_1(2t) & \text{ if } t\in [0,1/2];\\
\tilde{u}_2(2t-1) & \text{ if } t \in [1/2, 1].
\end{cases}
\end{equation*}
\end{proof}
\subsection{Approximation by nested maps}
In \cite{CL}*{Proposition 6.1}, Chambers and Liokumovich proved that if \(\{\Om_t\}_{t\in [0,1]}\) is a family of open sets and \(\ka>0\), \tes a nested family of open sets \(\{\tilde{\Om}_t\}_{t\in [0,1]}\) \st \(\tilde{\Om}_0\subset \Om_0\), \(\tilde{\Om}_1 \supset \Om_1\) and
\[\sup_{t\in [0,1]}\cH^n(\del\tilde{\Om}_t)\leq \sup_{t\in [0,1]}\cH^n(\del\Om_t)+\ka.\]
The following Proposition \ref{p.approx.by.nested} is the Allen-Cahn analogue of this theorem.
\begin{pro}\label{p.approx.by.nested}
Let \(\ph:[0,1]\ra \HN\) be a \cts map \st $|\ph(t)|\leq 1$ for all $t\in [0,1]$ and \(\ka>0\). Then \tes a nested map \(\ps:[0,1]\ra \HN\) \st \(\ps(0)\geq \ph(0)\), \(\ps(1)\leq \ph(1)\), $|\ps(t)|\leq 1$ for all $t\in [0,1]$ and
\[\sup_{t\in [0,1]}E_{\ve}(\ps(t))\leq \sup_{t\in [0,1]}E_{\ve}(\ph(t))+\ka.\]
\end{pro}
Before we prove Proposition \ref{p.approx.by.nested}, we need to prove few lemmas.
\begin{lem}\label{l.no.conc.mass}
	Let \(\ve>0\) and \(w:[0,1]\ra H^1(N)\) be a \cts map. Then, \fa \(\de>0\), \tes \(r>0\) \st
	\[\Ee{w(t),B(p,r)}\leq \de,\]
	\fa \(t\in [0,1]\) and \(p\in N\).
\end{lem}
\begin{proof}
	We assume by contradiction that \te \(\de>0\) and sequences \(\{t_i\}_{i=1}^{\infty}\subset [0,1]\) and \(\{p_i\}_{i=1}^{\infty}\subset N\) \st
	\begin{equation}\label{e.conc.mass}
	\Ee{w(t_i),B\left(p_i,i^{-1}\right)}> \de.
	\end{equation}
	Without loss of generality, we can assume that \(t_i\ra t_0\) and \(p_i\ra p_0\). Then, for all \(m\in \bbn\), if \(i\) is sufficiently large, \(B(p_i,i^{-1})\subset B(p_0,m^{-1})\). \tf by \eqref{e.conc.mass}, for all \(m\in \bbn\),
	\begin{align*}
	\Ee{w(t_0),B(p_0,m^{-1})}=\lim_{i\ra \infty}\Ee{w(t_i),B(p_0,m^{-1})}\geq \de.
	\end{align*}
	This contradicts the fact that
	\[\lim_{m\ra \infty}\Ee{w(t_0),B(p_0,m^{-1})}=0.\]
\end{proof}
\begin{lem}\label{l.coarea}
	Let \(u_0,u_1\in L^{\infty}(N)\) \w \(u_0\geq u_1\) and \(\md{u_0}, \md{u_1}\leq 1\). \sps \(h:\bbr\ra [-1,1]\) is a piecewise $C^1$ function so that $h'\in L^{\infty}(\bbr)$ and $e_{\ve}(h)$ is compactly supported inside the compact interval $[-a,a]$. For \(p\in N\) and \(r>0\), let \(b^{r}:N\ra \bbr\) be defined by \(b^r(x)=h(d_{p}(x)-r)\), where \(d_p(x)=d(x,p)\). Then, for all \(0< s_1< s_2\) and \(\ve>0\),
	\[\int_{ s_1}^{ s_2}E_{\ve}\left(b^r,\{u_0>b^r>u_1\}\right)\;dr\leq \int_{B(p,s_2+a)}\int_{\{u_0(x)>h>u_1(x)\}}e_{\ve}(h)(t)\;dt\;d\cH^{n+1}(x).\]
\end{lem}
\begin{proof} 
	\(d_p:N\ra \bbr\) is a Lipschitz \cts \fn \w \(\md{\na d_p}=1\), \(\cH^{n+1}\)-a.e. As a consequence,
	\begin{equation}
	\cH^{n+1}\left(\left\{d_p=s\right\}\right) = 0 \;\;\forall\;s\geq 0,\label{sph}
	\end{equation}
	since otherwise $\nabla d_p=0$ on a set of positive \(\cH^{n+1}\)-measure. For all $r>0$, $b^r$ is Lipschitz continuous; by \cite{GT}*{Theorem 7.8} and \eqref{sph},
	\begin{equation}
	\nabla b^r(x) = h'(d_p(x)-r)\nabla d_p(x)\;\text{ for } \cH^{n+1}\text{-a.e. }x\in N.
	\end{equation}
	\tf
	\begin{align}
	& \int_{ s_1}^{ s_2}E_{\ve}\left(b^r,\{u_0>b^r>u_1\}\right)\;dr\nonumber\\
	& =\int_{ s_1}^{ s_2}\int_{\{u_0>b^r>u_1\}}\left[\frac{\ve}{2}h'(d_p(x)-r)^2+\frac{1}{\ve}W\left(h(d_p(x)-r)\right)\right]\;d\cH^{n+1}(x)\;dr\nonumber\\
	&=\int_{ s_1}^{ s_2}\int_{-a}^{a}\left[\frac{\ve}{2}h'(t)^2+\frac{1}{\ve}W(h(t))\right]\cH^n\left(\{d_p-r=t\}\cap \{u_0>b^r>u_1\}\right)\;dt\;dr.\label{e.coarea}
	\end{align}
	In the last step we have used the co-area formula. It follows from the definition of \(b^r\) that for fixed \(r\) and \(t\),
	\[\{d_p-r=t\}\cap \{u_0>b^r>u_1\}=\{d_p=r+t\}\cap \{u_0>h(t)>u_1\}.\]
	Hence, by Fubini's theorem and the co-area formula, \eqref{e.coarea} implies that
	\begin{align*}
	& \int_{ s_1}^{ s_2}E_{\ve}\left(b^r,\{u_0>b^r>u_1\}\right)\;dr\\
	&=\int_{ s_1}^{ s_2}\int_{-a}^{a}\left[\frac{\ve}{2}h'(t)^2+\frac{1}{\ve}W(h(t))\right]\cH^n\left(\{d_p=r+t\}\cap \{u_0>h(t)>u_1\}\right)\;dt\;dr\\
	&\leq \int_{-a}^{a}e_{\ve}(h)(t)\;\cH^{n+1}\left(B(p,s_2+a)\cap\{u_0>h(t)>u_1\}\right)\;dt\\
	& = \int_{-a}^{a}e_{\ve}(h)(t)\int_{B(p,s_2+a)} \chi_{\{u_0>h(t)>u_1\}}(x)\;d\cH^{n+1}(x)\;dt\\
	& = \int_{B(p,s_2+a)}\int_{\{u_0(x)>h>u_1(x)\}}e_{\ve}(h)(t)\;dt\;d\cH^{n+1}(x).
	\end{align*}
\end{proof}
For \(\rho>0\), let \(h^{\rh}:\bbr\ra [-1,1]\) be defined by
\begin{equation}\label{e.h.rho}
h^{\rh}(t)=\begin{cases}
\frac{t}{\rh} &\text{ if } |t|\leq \rh;\\
1 & \text{ if } t \geq \rho;\\
-1& \text{ if } t\leq -\rh.
\end{cases}
\end{equation}
Setting \(h=h^{\rh}\) in Lemma \ref{l.coarea} and using the notation \(b^{\rh,r}(x)=h^{\rh}(d_p(x)-r)\), one obtains
\begin{equation}\label{e.coarea.specialized}
\int_{ s_1}^{ s_2}E_{\ve}\left(b^{\rh,r},\{u_0>b^{\rh,r}>u_1\}\right)\;dr\leq C(\ve,\rh)\nm{u_0-u_1}_{L^1(N)},
\end{equation}
where \(u_0,\;u_1\) are as in Lemma \ref{l.coarea} and 
\begin{equation}\label{C.ep.rh}
C(\ve,\rh)=\frac{\ve}{2\rh^2}+\frac{1}{\ve}\nm{W}_{L^{\infty}([-1,1])}.
\end{equation}
\begin{lem}\label{l.deformation}
	Let \(\ve,\de>0\). \sps \(u_0,u_1\in H^1(N)\cap L^{\infty}(N)\) \st \(u_0\geq u_1\); \(\md{u_0},\md{u_1}\leq 1\) and 
	\begin{equation}\label{e.energy.diff.u0.u1}
	\int_{N}\md{e_{\ve}(u_0)-e_{\ve}(u_1)}\leq \de.
	\end{equation}
	We fix a nested map \(w:[0,1]\ra H^1(N)\) \st $w(0)\equiv 1$, $w(1)\equiv -1$ and \(\md{w(t)}\leq 1\) \fa \(t \in [0,1]\). Let \(R>0\) be \st \fa \(p\in N\),
	\begin{equation}\label{e.energy.w.u0.u1}
	E_{\ve}(u_0,B(p,4R))\leq \de,\quad E_{\ve}(u_1,B(p,4R))\leq \de\quad\text{ and }\quad \sup_{t\in [0,1]}E_{\ve}(w(t),B(p,4R))\leq \de.
	\end{equation}
	\sps \(N\) can be covered by \(I\) balls of radius \(R\). Then, using the notation of \eqref{C.ep.rh},
	\begin{equation}\label{e.L^1.norm}
	\nm{u_0-u_1}_{L^1(N)}\leq\frac{\de R}{C(\ve, R)I}
	\end{equation}
	implies that \tes a nested map \(u:[0,1]\ra H^1(N)\) \st \(u(0)=u_0\), \(u(1)=u_1\), $|u(t)|\leq 1$ for all $t\in [0,1]$ and
	\[\sup_{t\in [0,1]}E_{\ve}(u(t))\leq \min\{E_{\ve}(u_0),E_{\ve}(u_1)\}+9\de.\]
\end{lem}
\begin{proof}
	Let 
	\begin{equation}\label{e.cover}
	N=\bigcup_{i=1}^I B(p_i,R).
	\end{equation}
	For $r>0$ and \(i=1,2,\dots,I\), let \(b_i^r:N\ra [-1,1]\) be defined by (using the notation as in \eqref{e.h.rho})
	\[b_i^r(x)=h^R(d_{p_i}(x)-r).\]
	We inductively define a sequence \(\{v_i\}_{i=0}^I\) with 
	\[v_0\geq v_1\geq \dots \geq v_I\]
	as follows. Set \(v_0=u_0\).
	Let us assume that \(v_k\) has been defined for \(1\leq k\leq i-1\) and
	\[u_0=v_0\geq v_1\geq\dots\geq v_{i-1}\geq u_1,\]
	which implies that
	\begin{equation}\label{e.L^1.norm'}
	\nm{v_{i-1}-u_1}_{L^1(N)}\leq\nm{u_0-u_1}_{L^1(N)}.
	\end{equation}
	\tf using \eqref{e.coarea.specialized}, \eqref{e.L^1.norm'} and \eqref{e.L^1.norm},
	\[\int_{2R}^{3R}E_{\ve}\left(b_i^r,\{v_{i-1}>b_i^r>u_1\}\right)\;dr\leq \frac{\de R}{I}.\]
	So \tes \(r_i\in (2R,3R)\) \st
	\begin{equation}\label{e.energy.b_i}
	E_{\ve}(b_i^{r_i},\{v_{i-1}>b_i^{r_i}>u_1\}) \leq \frac{\de}{I}.
	\end{equation}
	We define
	\begin{equation}\label{e.def.v_i}
	v_i=\min\{v_{i-1},\max\{u_1,b_i^{r_i}\}\}=\begin{cases}
	v_{i-1} &\text{ on }\{b_i^{r_i}\geq v_{i-1}\};\\
	b_i^{r_i} &\text{ on }\{v_{i-1}>b_i^{r_i}>u_1\};\\
	u_1&\text{ on }\{u_1\geq b_i^{r_i}\}.
	\end{cases}
	\end{equation}
	Then \(v_{i-1}\geq v_i\geq u_1\). 
	
	Using the definition of \(v_i\) in \eqref{e.def.v_i}, one can prove by induction that for each \(1\leq i\leq I\), \te pairwise disjoint, \(\cH^{n+1}\)-measurable sets \(G^0_i\), \(G^1_i\), \(\{G_{k,i}\}_{k=1}^i\) with
	\begin{equation}\label{N.partn}
	N=G^0_i\cup G^1_i\cup\left(\bigcup_{k=1}^iG_{k,i}\right)
	\end{equation}
	 \st the following conditions are satisfied.
	\begin{itemize}
		\item[(i)]
		\begin{equation}\label{e.v_i}
		v_i=\begin{cases}
		u_0 &\text{ on } G^0_i; \\
		u_1&\text{ on }G^1_i;\\
		\bk & \text{ on }G_{k,i},\; 1\leq k\leq i.
		\end{cases}
		\end{equation}
		\item[(ii)] 
		\begin{equation}\label{e.G_ii}
		G_{i,i}=\{v_{i-1}>\bi>u_1\}.
		\end{equation}
		\item[(iii)] For \(1<i\leq I\) and \(1\leq k \leq i-1\),
		\begin{equation}\label{e.set.incl}
		G^0_{i-1}\supset G^0_{i};\quad G^1_{i-1}\subset G^1_i;\quad G_{k,i-1}\supset G_{k,i}.
		\end{equation}
		\item[(iv)] 
		\begin{equation}\label{e.u1}
		\bigcup_{k=1}^iB(p_k,R)\subset G^1_i
		\end{equation}
		(For this item one needs to use the fact that \(\bk\equiv -1\) on \(B(p_k,R)\).)
	\end{itemize}
	\eqref{e.u1} implies that (by \eqref{e.cover}) \(v_I=u_1.\) Moreover,
	\begin{align}
	\Ee{v_i} &= \Ee{u_0,G_i^0}+\Ee{u_1,G^1_i}+\sum_{k=1}^{i}\Ee{\bk,G_{k,i}}\;(\text{by }\eqref{N.partn},\; \eqref{e.v_i})\nonumber\\
	&\leq \min\{\Ee{u_0,N},\Ee{u_1,N}\}+\int_N|e_{\ve}(u_0)-e_{\ve}(u_1)|+\sum_{k=1}^{i}\Ee{\bk,G_{k,k}}\;(\text{by }\eqref{e.set.incl})\nonumber\\
	&\leq \min\{\Ee{u_0},\Ee{u_1}\}+2\de\;(\text{by } \eqref{e.energy.diff.u0.u1},\;\eqref{e.G_ii},\;\eqref{e.energy.b_i}). \label{e.energy.v_i}
	\end{align}
	Similarly, using \eqref{N.partn} and \eqref{e.v_i}, for any \(p\in N\),
	\begin{align}
	&\Ee{v_i,B(p, 4R)}\nonumber\\
	&= \Ee{u_0,G_i^0\cap B(p, 4R)}+\Ee{u_1,G^1_i\cap B(p, 4R)}+\sum_{k=1}^{i}\Ee{\bk,G_{k,i}\cap B(p, 4R)}\nonumber\\
	&\leq 3\de\;(\text{by } \eqref{e.energy.w.u0.u1},\;\eqref{e.set.incl},\;\eqref{e.G_ii},\;\eqref{e.energy.b_i}).\label{e.energy.b_i.ball}
	\end{align}
	For each \(1\leq i\leq I\), we define \(\be_i:[0,1]\ra H^1(N)\) by 
	\begin{equation*}
	\be_i(t)=\min\{v_{i-1},\max\{v_i,w(t)\}\}=\begin{cases}
	v_{i-1} &\text{ on }\{w(t)\geq v_{i-1}\};\\
	w(t) &\text{ on }\{v_{i-1}>w(t)>v_i\};\\
	v_i&\text{ on }\{v_i\geq w(t)\}.
	\end{cases}
	\end{equation*}
Here \(w:[0,1]\ra H^1(N)\) is as stated in the Lemma \ref{l.deformation}. Since \(w\) is nested, \(\be_i\) is also nested. It follows from \eqref{e.def.v_i} that \(|v_k|\leq 1\) \fa \(0\leq k\leq I\). Hence $|\be_i(t)|\leq 1$ for all $t\in [0,1]$. Moreover, \(w(0)\equiv 1\) (resp. \(w(1)\equiv -1\)) implies that \(\be_i(0)=v_{i-1}\) (resp. \(\be_i(1)=v_{i}\)). Using the fact that \(\bi\equiv 1\) on \(N\setminus B(p_i,4R)\), it also follows from \eqref{e.def.v_i} that \(v_{i-1}=v_i\) on \(N\setminus B(p_i,4R)\); hence \fa \(t\in [0,1]\),
	\begin{equation*}
	\be_i(t)=v_{i-1}=v_i\text{ on }N\setminus B(p_i, 4R).
	\end{equation*}
	\tf \fa \(t\in [0,1]\),
	\begin{align*}
	\Ee{\be_i(t)}&\leq \Ee{v_i,N}+\Ee{v_{i-1},B(p_i,4R)}+\Ee{v_{i},B(p_i,4R)}+\Ee{w(t),B(p_i,4R)}\\
	&\leq \min\{\Ee{u_0},\Ee{u_1}\}+9\de \;(\text{by } \eqref{e.energy.v_i},\;\eqref{e.energy.b_i.ball},\;\eqref{e.energy.w.u0.u1}).
	\end{align*}
	Finally we obtain the required map \(u:[0,1]\ra H^1(N)\) by concatenating all the maps \(\be_i\), \(i=1,2,\dots,I\).
\end{proof}

\begin{proof}[Proof of Proposition \ref{p.approx.by.nested}]
We fix a nested map \(w_0:[0,1]\ra H^1(N)\) \st $w_0(0)\equiv 1$, $w_0(1)\equiv -1$ and \(\md{w_0(t)}\leq 1\) \fa \(t \in [0,1]\). (For instance, one can define \(w_0(t)\) to be equal to the constant function $1-2t$.) Let $\de_0=\ka/9$ and
\[A_0=\sup_{t\in [0,1]}E_{\ve}(\ph(t)).\]
By Lemma \ref{l.no.conc.mass}, \tes \(R_0>0\) \st \fa \(p\in N\) and $t\in [0,1],$
\begin{equation}\label{e.energy.w_0}
E_{\ve}(\ph(t),B(p,4R_0))\leq \frac{\de_0}{2}\quad\text{ and }\quad E_{\ve}(w_0(t),B(p,4R_0))\leq \de_0. 
\end{equation}
\sps \(N\) can be covered by \(I_0\) balls of radius \(R_0\). One can choose \(m\in\bbn\) \st \(|t_1-t_2|\leq 1/m\) implies
\begin{equation}\label{e.ps.1}
\int_N|e_{\ve}(\ph(t_1))-e_{\ve}(\ph(t_2))|\leq \frac{\de_0}{2}
\end{equation}
and 
\begin{equation}\label{e.ps.2}
\|\ph(t_1)-\ph(t_2)\|_{L^1(N)}\leq \frac{\de_0 R_0}{C(\ve,R_0)I_0}.
\end{equation}

We define the sequence \(\{\hat{\ph}_i\}_{i=0}^{2m}\) by setting \(\phh_{2k}=\ph(k/m)\) and
\begin{equation*}
\phh_{2k+1}=\min\{\phh_{2k}, \phh_{2k+2}\}=\begin{cases}
\phh_{2k}&\text{ on }\{\phh_{2k}\leq \phh_{2k+2}\};\\
\phh_{2k+2} &\text{ on }\{\phh_{2k} > \phh_{2k+2}\}.
\end{cases}
\end{equation*}
Hence, \eqref{e.ps.1} implies that for \(0\leq k \leq m-1\),
\begin{align}
&\int_N|e_{\ve}(\phh_{2k})-e_{\ve}(\phh_{2k+1})|= \int_{\{\phh_{2k}>\phh_{2k+2}\}}|e_{\ve}(\phh_{2k})-e_{\ve}(\phh_{2k+2})|\leq \frac{\de_0}{2}.\label{e.tht1}
\end{align}
Similarly, \eqref{e.ps.2} implies that for \(0\leq k \leq m-1\),
\[\|\phh_{2k}-\phh_{2k+1}\|_{L^1(N)} \leq \frac{\de_0 R_0}{C(\ve,R_0)I_0}.\]
By \eqref{e.energy.w_0} and \eqref{e.tht1}, \fa \(0\leq i\leq 2m\) and $p\in N$,
\[E_{\ve}(\phh_i,B(p,4R_0))\leq\de_0.\]

By Lemma \ref{l.deformation}, for \(0\leq i\leq m-1\), \tes a nested map \(\ga_i:[0,1]\ra \HN\) \st \(\ga_i(0)=\phh_{2i}\), \(\ga_i(1)=\phh_{2i+1}\), \(|\ga_i(t)|\leq 1\) for all \(t\in [0,1]\) and 
\[\sup_{t\in [0,1]} E_{\ve}(\ga_i(t))\leq E_{\ve}(\phh_{2i})+\ka\leq A_0+\ka.\]
\tf
\[\ga_i(1)=\phh_{2i+1}\leq \phh_{2i+2}=\ga_{i+1}(0).\]
One obtains the map \(\ps\) in Proposition \ref{p.approx.by.nested} from the maps \(\{\ga_i\}_{i=0}^{m-1}\) by repeatedly applying Lemma \ref{l.concatenation}. More precisely, setting \(u_1=\ga_0\) and \(u_2=\ga_1\) in Lemma \ref{l.concatenation}, we get a nested map \(\bar{\ga}_1:[0,1]\ra \HN\) \st
\[\bar{\ga}_1(0)\geq \hat{\ph}_0,\;\gab_1(1)\leq \hat{\ph}_3,\; \sup
_{t\in [0,1]}\nm{\gab_1(t)}_{L^{\infty}(N)}\leq 1\;\text{ and }\; \sup_{t\in [0,1]} E_{\ve}(\gab_1(t))\leq A_0+\ka.\]
Let us assume that \tes a nested map \(\gab_i:[0,1]\ra \HN\), \(1\leq i<m-1\), \st
\begin{align}
&\bar{\ga}_i(0)\geq \hat{\ph}_0,\quad\gab_i(1)\leq \hat{\ph}_{2i+1}\leq \phh_{2i+2}=\ga_{i+1}(0),\\
\sup_{t\in [0,1]}&\nm{\gab_i(t)}_{L^{\infty}(N)}\leq 1 \;\text{ and }\; \sup_{t\in [0,1]} E_{\ve}(\gab_i(t))\leq A_0+\ka.
\end{align}
Then choosing \(u_1=\gab_i\) and \(u_2=\ga_{i+1}\) in Lemma \ref{l.concatenation}, one gets a nested map \(\bar{\ga}_{i+1}:[0,1]\ra \HN\) \st
\[\bar{\ga}_{i+1}(0)\geq \hat{\ph}_0,\;\gab_{i+1}(1)\leq \hat{\ph}_{2i+3},\;\sup_{t\in [0,1]}\nm{\gab_{i+1}(t)}_{L^{\infty}(N)}\leq 1\; \text{ and }\; \sup_{t\in [0,1]} E_{\ve}(\gab_{i+1}(t))\leq A_0+\ka.\]
The map \(\ps\) in Proposition \ref{p.approx.by.nested} is obtained by setting \(\ps=\gab_{m-1}\).
\end{proof}
\section{A deformation lemma}
The following Lemma \ref{l.deformation.a} is motivated by \cite{CL}*{Lemma 7.1 (3)}. To prove this lemma, we adapt the argument of Chambers and Liokumovich \cite{CL}*{Proof of Lemma 7.1} in the Allen-Cahn setting.
\begin{lem}\label{l.deformation.a}
Let \(N\) be a closed \Rm \mf and \(\Om \subset N\) be an open set \w smooth boundary \(\del\Om\). \sps \(f:N\ra [1/3,\infty)\) is a Morse \fn so that in the interval $[1/3,2/3]$, $f$ has no critical value which is a non-global local maxima or minima;
\[\min_N f=1/3;\quad \max_N f>1;\quad \Om\subset\subset f^{-1}\left([1/3,2/3)\right).\]
We set 
$$\tilde{\Om}=f^{-1}\left([1/3,1]\right).$$
Then, for all \(\et>0\), \te \(\ve_1,\;\tilde{\et}>0\), depending on \(\et,\;\Om,\;\tilde{\Om},\;f\big|_{\tilde{\Om}}\), \st the following two conditions are satisfied.
\begin{itemize}
	\item[(i)] If \(0<\ve\leq \ve_1\) and \(u_0\in \HN\) satisfies \(|u_0|\leq 1\), \(\|1-u_0\|_{L^1(\Om)}\leq\tilde{\et}\), then \tes \(u:[0,1]\ra \HN\) \st \(u(0)=u_0\), \(u(1)\big|_{\Om}\equiv 1\) and 
	\[\sup_{t\in [0,1]}\Ee{u(t)}\leq\Ee{u_0}+2\si\cH^n(\del\Om)+\et.\]
	\item[(ii)] If \(0<\ve\leq \ve_1\) and \(u_0\in \HN\) satisfies  \(|u_0|\leq 1\), \(\|1+u_0\|_{L^1(\Om)}\leq\tilde{\et}\), then \tes \(u:[0,1]\ra \HN\) \st \(u(0)=u_0\), \(u(1)\big|_{\Om}\equiv -1\) and 
	\[\sup_{t\in [0,1]}\Ee{u(t)}\leq\Ee{u_0}+2\si\cH^n(\del\Om)+\et.\]
\end{itemize}
\end{lem}
\begin{proof}
Let \(q: \bbr \ra \bbr\) be the unique solution of the following ODE.
\begin{equation}\label{e.ode}
\vp'(t)= \sqrt{2W(\vp(t))};\quad \vp(0)=0.
\end{equation}
For all \(t \in \bbr\), \(-1<q(t)<1\) and 
\begin{equation}\label{e.exp.decay}
\text{as } t \ra \pm \infty,\; (q(t)\mp 1) \text{ converges to zero exponentially fast.}
\end{equation}
\(q_{\ve}(t)=q(t/\ve)\) is a solution of the one dimensional Allen-Cahn equation
\[\ve^2\vp''(t)=W'(\vp(t))\]
with finite total energy:
\begin{equation}\label{e.energy.1D.soln}
\int_{-\infty}^{\infty}\left[\frac{\ve}{2}\left(q'_{\ve}(t)\right)^2+\frac{1}{\ve}W(q_{\ve}(t))\right] dt=2\si.
\end{equation}
For $\ve >0$, we define Lipschitz continuous function 
\begin{equation}\label{def.qte}
\qte(t)=
\begin{cases}
\qe(t) & \text{if } |t|\leq \sqrt{\ve};\\
\qe(\se)+\left(\frac{t}{\se}-1\right)(1-\qe(\se)) &\text{if } \se \leq t \leq 2 \se;\\
1 &\text{if } t\geq 2\se;\\
\qe(-\se)+ \left(\frac{t}{\se}+1\right)(1+\qe(-\se)) &\text{if } -2\se \leq t \leq -\se;\\
-1 &\text{if } t\leq -2\se.
\end{cases}
\end{equation}
For \(t\in \bbr\) and \(x\in N\), we set
\begin{equation}\label{e.d.del.omega}
d_1^t(x)=d_{\del \Om}(x)-t,\quad \text{ where }\quad d_{\del \Om}(x)=\begin{cases}
-d(x,\del \Om) & \text{ if } x\in \Om; \\
d(x,\del\Om) & \text{ if } x\notin\Om.
\end{cases}
\end{equation}
For \(t\in [1/3,1]\) and \(x\in N\), we set
\begin{equation}
d_2^t(x)=\begin{cases}
-d(x,f^{-1}(t)) & \text{ if } f(x)\leq t;\\
d(x,f^{-1}(t)) & \text{ if } f(x)\geq t.
\end{cases}
\end{equation}
Following \cite{G}*{Section 7 and Section 9}, we define the \cts maps \(w_{1,\ve}:\bbr\ra \HN\) and \(w_{2,\ve}:[0,1]\ra H^1(N)\) by 
\begin{equation}\label{e.def.w.1.eps}
w_{1,\ve}(t)=\qte\circ d_1^t;
\end{equation}
\begin{equation}\label{w.2.eps}
w_{2,\ve}(t)=\begin{cases}
\hat{q}_{\ve}\circ d^t_2 & \text{if } \frac{1}{3}\leq t \leq \frac{2}{3}; \\
1-3t(1-w_{2,\ve}(1/3)) & \text{if } 0 \leq t \leq \frac{1}{3};\\
-1+3(1-t)(1+w_{2,\ve}(2/3)) & \text{if } \frac{2}{3}\leq t\leq 1.
\end{cases}
\end{equation}
Since in the interval $[1/3,2/3]$, $f$ has no critical value which is a non-global local maxima or minima, \(t\mapsto f^{-1}(t)\) is \cts on \([1/3,2/3]\) in the Hausdorff topology. This implies that \(w_{2,\ve}\) is \cts (see \cite{G}*{Proposition 9.2}).

Let us fix \(\et>0\). From the argument in \cite{G}*{Section 9}, it follows that \te \(\ve',\;t_0>0\), depending on $\et$, \(\del \Om\) and \(\tilde{\Om}\), \st if \(0<\ve\leq\ve'\) and \(|t|\leq 2t_0\) then 
\begin{equation}\label{e.energy.w.1.eps}
\Ee{w_{1,\ve}(t)}\leq 2\si\cH^n(\del \Om)+\frac{\et}{2}.
\end{equation}
By the ``no concentration of mass" property (\cite{MN_ricci_positive}*{Lemma 5.2}), \tes \(0<R<t_0/5\), depending on \(\et,\;\Om,\;\tilde{\Om}\) and \(f\big|_{\tilde{\Om}}\),  \st
\begin{equation}
\cH^n\left(f^{-1}(t)\cap B(p,5R)\right)\leq \frac{1}{2\si}\frac{\et}{3},
\end{equation}
\fa \(t\in [1/3,2/3]\) and \(B(p,5R)\subset \Om\). Moreover, \(w_{2,\ve}(1/3)>0\) on \(N\) and \(w_{2,\ve}(2/3)<0\) on \(\Om\). As a consequence, by the results in \cite{G}*{Section 9}, \tes \(0<\ve''\leq R^2/4\), depending on \(\Om,\;\tilde{\Om}\) and \(f\big|_{\tilde{\Om}}\)  \st
\begin{equation}\label{e.energy.w.2.eps}
\Ee{\wte, B(p,4R)}\leq \frac{\et}{2},
\end{equation}
\fa \(t\in [0,1]\) and \(B(p,5R)\subset \Om\).
We define \(\ve_1=\min\{\ve',\ve''\}.\) By our definitions of \(R\) and \(\ve''\),
\begin{equation}\label{e.eps.1.bound}
2\sqrt{\ve_1}\leq R<\frac{t_0}{5}.
\end{equation}

Let us fix \(\ve\in (0,\ve_1]\). Using the notation of \eqref{e.d.del.omega}, let 
\begin{equation}\label{Om_r}
\Om_r=\{x\in N:d_{\del \Om}(x)\leq r\}.
\end{equation}
For \(r>0\) and for a fixed \(p\in\Om_{-t_0}\), we define \(\om^r_{\ve}:N\ra \bbr\) by
\[\om^r_{\ve}(x)=\qte\left(r-d_p(x)\right),\]
where \(d_p(x)=d(x,p).\) 

By \eqref{e.eps.1.bound}, \(B\left(p,3R+2\sqrt{\ve_1}\right)\subset \Om\). \tf it follows from Lemma \ref{l.coarea} that for \(u\in L^{\infty}(N)\) \w \(|u|\leq 1\),
\begin{align}
&\int_{2R}^{3R}\Ee{\om^r_{\ve},\{1>\om^r_{\ve}>u\}}\;dr\\
&\leq \int_{\Om}\int_{\{1>\qte>u(x)\}}e_{\ve}(\qte)(t)\;dt\;d\cH^{n+1}(x)\quad(\text{since }\qte\text{ is an odd function})\\
&=\int_{\Om}\int_{u(x)}^{1}\left[\frac{\ve}{2}\qte'\left(\qte^{-1}(s)\right)+\frac{1}{\ve}\frac{W(s)}{\qte'\left(\qte^{-1}(s)\right)}\right]\;ds\;d\cH^{n+1}(x).\label{e.integrand}
\end{align}
In \eqref{e.integrand}, \(\qte\) is thought of as a bijective map from \(\left[-2\sqrt{\ve},2\sqrt{\ve}\right]\) to \([-1,1]\). We claim that \tes \(C_0=C_0(W,\ve_1)>0\) \st the \(L^{\infty}([-1,1])\) norm of the integrand in \eqref{e.integrand} is bounded by \(C_0\). (In particular, \(C_0\) does not depend on $\ve$.) Indeed, the integrand is a non-negative, even function. If \(0\leq t<\sqrt{\ve},\) then
\begin{align}
& \frac{\ve}{2}\hat{q}'_{\ve}(t)+\frac{1}{\ve}\frac{W(\qte(t))}{\qte'(t)}\\
&=\frac{1}{2}q'\left(t/\ve\right)+\frac{W(q(t/\ve))}{q'(t/\ve)}\\
&=\sqrt{2W(q(t/\ve))}\; (\text{by }\eqref{e.ode})\\
&\leq\sqrt{2}\|W\|_{L^{\infty}([-1,1])}.\label{e.clm.1}
\end{align}
If \(\sqrt{\ve}<t<2\sqrt{\ve}\), using the fact that
\[\sup_{t\in [-1,1]}\frac{W(t)}{(1-t)^2}=C_1<\infty,\]
we obtain
\begin{align}
& \frac{\ve}{2}\qte'(t)+\frac{1}{\ve}\frac{W(\qte(t))}{\qte'(t)}\\
&\leq \frac{\sqrt{\ve}}{2}\left(1-q\left(\ve^{-1/2}\right)\right)+\frac{1}{\sqrt{\ve}}\frac{W\left(q\left(\ve^{-1/2}\right)\right)}{\left(1-q\left(\ve^{-1/2}\right)\right)}\\
&\leq \left(\frac{\sqrt{\ve}}{2}+\frac{C_1}{\sqrt{\ve}}\right)\left(1-q\left(\ve^{-1/2}\right)\right).\label{e.clm.2}
\end{align}
By \eqref{e.exp.decay}, the expression in \eqref{e.clm.2} is bounded by some constant \(C_2=C_2(\ve_1,C_1)\). Thus our claim follows from \eqref{e.clm.1} and \eqref{e.clm.2}. \eqref{e.integrand}, together with the claim, implies that
\begin{equation}\label{e.coarea'}
\int_{2R}^{3R}\Ee{\om^r_{\ve},\{1>\om^r_{\ve}>u\}}\;dr\leq C_0 \|1-u\|_{L^1(\Om)}.
\end{equation}

We choose a covering 
\begin{equation}\label{e.cover'}
\Om_{-t_0}=\bigcup_{i=1}^IB(p_i,R);
\end{equation}
each \(p_i\in \Om_{-t_0}\) so that (by \eqref{e.eps.1.bound}) \(B(p_i,5R)\subset \Om.\)
To prove part (i) of Lemma \ref{l.deformation.a}, we set 
\begin{equation}\label{e.def.eta.tilde}
\tilde{\et}=\frac{\et R}{2C_0I},
\end{equation}
where $C_0$ is as in the above claim. Let \(u_0\) be as in the statement of Lemma \ref{l.deformation.a}, part (i). We inductively define a sequence \(\{v_i\}_{i=0}^I\), 
\[-1\leq v_0\leq v_1\leq\dots\leq v_I\leq 1,\]
as follows. Set \(v_0=u_0.\) \sps \(v_k\) has been defined for \(0\leq k\leq i-1\) so that
\[u_0=v_0\leq v_1\leq\dots \leq v_{i-1}\leq 1;\]
hence
\begin{equation}\label{e.1-v_i}
\|1-v_{i-1}\|_{L^1(\Om)}\leq \|1-u_0\|_{L^1(\Om)}\leq \tilde{\et}.
\end{equation}
Let \(\om_{i,\ve}^r(x)=\qte\left(r-d_{p_i}(x)\right)\). By \eqref{e.coarea'}, \eqref{e.1-v_i} and \eqref{e.def.eta.tilde}, \tes \(r_i\in (2R,3R)\) \st
\begin{equation}\label{e.energy.omega_i}
\Eee{\om_{i,\ve}^{r_i},\{v_{i-1}<\om_{i,\ve}^{r_i}<1\}}\leq \frac{\et}{2I}.
\end{equation}
We define
\begin{equation}\label{e.def.v_i.bis}
v_i=\max \{v_{i-1},\om_{i,\ve}^{r_i}\}=\begin{cases}
v_{i-1} &\text{on } \{v_{i-1}\geq \om_{i,\ve}^{r_i}\};\\
\om_{i,\ve}^{r_i} &\text{on } \{v_{i-1}< \om_{i,\ve}^{r_i}\}.
\end{cases}
\end{equation}
It follows from \eqref{e.def.v_i.bis} that \(v_{i-1}\leq v_i\leq 1\). Moreover, since \(\om_{i,\ve}^{r_i}\equiv -1\) on \(N\setminus B(p_i,4R)\),
\begin{equation}\label{e.v_i=v_{i-1}}
v_i=v_{i-1}\text{ on } N\setminus B(p_i,4R).
\end{equation}
Thus we have obtained the sequence \(\{v_i\}_{i=0}^I\) with $v_0=u_0$ and we set \(v_I=\bar{u}\). As \(\om_{i,\ve}^{r_i}\equiv 1\) on $B(p_i,R)$, using \eqref{e.def.v_i.bis} and \eqref{e.cover'}, one can prove by induction that
\begin{equation}\label{e.u.bar.1}
\bar{u}\big|_{\Om_{-t_0}}\equiv 1.
\end{equation}

By \eqref{e.def.v_i.bis} and \eqref{e.energy.omega_i}, 
\begin{equation}\label{v_i.neq v_{i-1}}
\Eee{v_i,\{v_i\neq v_{i-1}\}}=\Eee{\om_{i,\ve}^{r_i},\{v_{i-1}<\om_{i,\ve}^{r_i}<1\}}\leq \frac{\et}{2I}.
\end{equation}
Thus
\begin{align}
&\Ee{v_i}=\Ee{v_{i-1}}+\Ee{v_i,\{v_i\neq v_{i-1}\}}\leq \Ee{v_{i-1}}+\frac{\et}{2I}\\
&\implies \Ee{v_i}\leq \Ee{u_0}+\frac{\et i}{2 I}\quad \forall \;0\leq i\leq I.\label{e.energy.v_i'}
\end{align}
For \(1\leq i\leq I\), let \(\be_i:[0,1]\ra \HN\) be defined by
\begin{equation}\label{e.beta}
\be_i(t)=\max\{v_{i-1},\min\{v_i,-w_{2,\ve}(t)\}\}=\begin{cases}
v_{i-1} &\text{on } \{-w_{2,\ve}(t)\leq v_{i-1}\};\\
-w_{2,\ve}(t)&\text{on } \{v_{i-1}<-w_{2,\ve}(t)<v_i\};\\
v_i &\text{on } \{v_i\leq -w_{2,\ve}(t)\}.\\
\end{cases}
\end{equation}
Since \(w_{2,\ve}(0)\equiv 1\) and \(w_{2,\ve}(1)\equiv -1\), \(\be_i(0)=v_{i-1}\) and \(\be_i(1)=v_i\). Moreover, by \eqref{e.v_i=v_{i-1}},
\begin{equation}
\be_i(t)=v_{i-1}=v_i\text{ on } N\setminus B(p_i,4R).
\end{equation}
\tf using \eqref{e.beta}, \eqref{e.energy.v_i'}, \eqref{v_i.neq v_{i-1}} and \eqref{e.energy.w.2.eps}, one obtains
\begin{align}
\Ee{\be_i(t)}&\leq \Ee{v_{i-1}}+\Ee{v_i,\{v_i\neq v_{i-1}\}}+\Ee{-w_{2,\ve}(t),B(p_i,4R)}\\
&\leq \Ee{u_0}+\frac{\et (i-1)}{2 I}+\frac{\et }{2 I}+\frac{\et}{2}\\
&\leq\Ee{u_0}+\et.\label{e.energy.beta_i}
\end{align}
Concatenating all the \(\beta_i\)'s we get a map \(\be':[0,1]\ra \HN\) \st \(\be'(0)=u_0\), \(\be'(1)=\bar{u}=v_I\) and
\begin{equation}
\sup_{t\in [0,1]}\Ee{\be'(t)}\leq \Ee{u_0}+\et.\label{e.energy.beta'}
\end{equation}
Let \(\be'':[-2t_0,t_0]\ra \HN\) be defined by
\begin{equation}
\be''(t)=\max\{\bar{u},-w_{1,\ve}(t)\}.
\end{equation}
Then \(\be''(-2t_0)=\bar{u}\) as
\(\bar{u}\equiv 1\text{ on }\Om_{-t_0}\;(\eqref{e.u.bar.1})\text{ and } w_{1,\ve}(-2t_0)\equiv 1 \text{ on }N\setminus\Om_{-t_0}\;(\text{by }\eqref{e.eps.1.bound}).\)
Moreover, \(\be''(t_0)\big|_{\Om}\equiv 1\) as by \eqref{e.eps.1.bound}, \(w_{1,\ve}(t_0)\big|_{\Om}\equiv -1\). Using \eqref{e.energy.w.1.eps} and \eqref{e.energy.v_i'}, we conclude that \fa \(t \in [-2t_0,t_0]\),
\[\Ee{\be''(t)}\leq \Ee{\bar{u}}+\Ee{-w_{1,\ve}(t)}\leq \Ee{u_0}+2\si\cH^n(\del \Om)+\et.\]
Finally, the map \(u\) in Lemma \ref{l.deformation.a} part (i) is obtained by concatenating \(\be'\) and \(\be''\). This finishes the proof of part (i) of Lemma \ref{l.deformation.a}; part (ii) of the lemma can be deduced from part (i) by replacing \(u_0\) by \(-u_0\).
\end{proof}
The next lemma is motivated by the properties of the isoperimetric profile of a compact \Rm manifold (see \cite{CL}*{Lemma 7.1 (2)}). 
\begin{lem}\label{l.deformation.b}
Let \(\Om\) be a compact \Rm \mf (not necessarily closed). For all \(\et_1>0\), \te \(\ve_2,\;\et_2>0\), depending on \(\Om\) and \(\et_1\), \st the following holds. If \(0<\ve\leq \ve_2\) and \(u\in H^1(\Om)\) satisfies \(\|u\|_{L^{\infty}(\Om)}\leq 1\) and \(\min\{\|1-u\|_{L^1(\Om)}, \|1+u\|_{L^1(\Om)}\} > \et_1\), then \(E_{\ve}(u,\Om)> \et_2\).
\end{lem}
\begin{proof}
We assume by contradiction that \te sequences \(\{u_i\}_{i=1}^{\infty}\subset H^1(\Om)\) \w \(|u_i|\leq 1\) \fa \(i\) and \(\{\al_{i}\}_{i=1}^{\infty}\subset (0,\infty)\) \w \(\al_i \ra 0\) \st 
\begin{equation}\label{L^1}
\min\{\|1-u_i\|_{L^1(\Om)}, \|1+u_i\|_{L^1(\Om)}\} > \et_1\quad \forall\; i\in \bbn,
\end{equation}
and
\begin{equation}\label{E.al.1}
E_{\al_i}(u_i,\Om)\ra 0.
\end{equation}
Let \(F:[-1,1]\ra [-\si/2,\si/2]\) be as defined in \eqref{e.def.F} and \(v_i=F\circ u_i.\) As argued in \cite{HT}*{Section 2.1}, for all \(i\), \(|v_i|\leq \si/2\) and
\begin{equation}\label{E.al.2}
\int_{\Om}\md{\na v_i}\leq \frac{1}{2}E_{\al_i}(u_i,\Om).
\end{equation}
\tf \tes a subsequence \(\{v_{i_k}\}\subset \{v_i\}\) and \(v_{\infty}\in BV(\Om)\) \st 
\begin{equation}
v_{i_k} \ra v_{\infty}\text{ in } L^1(\Om) \text{ and pointwise a.e.}
\end{equation}
and (using \eqref{E.al.1} and \eqref{E.al.2})
\begin{equation}
\int_{\Om}\md{Dv_{\infty}}\leq \liminf_{k\ra \infty}\int_{\Om}\md{\na v_{i_k}}=0.
\end{equation}
Thus \(v_{\infty}\) is a constant function. Denoting \(u_{\infty}=F^{-1}(v_{\infty})\), by the dominated convergence theorem,
\begin{equation}
u_{i_k} \ra u_{\infty}\text{ pointwise a.e. and in } L^1(\Om).
\end{equation}
Moreover, by \eqref{E.al.1},
\[\int_{\Om}W(u_{\infty})= \lim_{k\ra \infty}\int_{\Om}W(u_{i_k})=0.\]
Since \(u_{\infty}\) is a constant function, either \(u_{\infty}\equiv 1\) or \(u_{\infty}\equiv -1\). However, this contradicts the assumption \eqref{L^1}.
\end{proof}
\section{Proof of Theorem \ref{t:main.thm}}
Let \(N\) be a closed \Rm manifold and \(\Om\subset N\) be an open set with smooth boundary \(\del \Om\). For \(\ve>0\), the \textit{\(\ve\)-Allen-Cahn width} of \(\Om\), which we denote by \(\la_{\ve}(\Om)\), is defined as follows \cite{G}. Let \(\sA\) be the set of all \cts maps \(\ze:[0,1]\ra H^1(\Om)\) \st \(\ze(0)\equiv 1\) and $\ze(1)\equiv-1$. 
Then 
\begin{equation}
\la_{\ve}(\Om)=\inf_{\ze\in \sA}\sup_{t\in [0,1]}\Ee{\ze(t),\Om}.
\end{equation}
It follows from \cite{G}*{Section 8} that
\begin{equation}\label{e.width.ineq.1}
\mathbb{W}(\Om)\leq \frac{1}{2\si}\liminf_{\ve\ra 0^{+}}\la_{\ve}(\Om),
\end{equation}
where \(\mathbb{W}(\Om)\) is as defined in \eqref{e.def.rel.width}. Motivated by \cite{CL}*{Section 2.2}, we also make the following definition.
\begin{defn}
Let \(\sB\) be the set of all \cts maps \(\ze:[0,1]\ra H^1(N)\) \st \(\ze(0)\big|_{\Om}\equiv 1\) and $\ze(1)\big|_{\Om}\equiv-1$. For $\ve>0$, we define
\begin{equation}\label{la.tld.ep}
\tilde{\la}_{\ve}(\Om)= \inf_{\ze\in \sB}\sup_{t\in [0,1]}\Ee{\ze(t),N}.
\end{equation}
\end{defn}
Given \(\tilde{\ze}\in \sB\) one can define \(\ze\in \sA\) by \(\ze(t)=\tilde{\ze}(t)\big|_{\Om};\)
hence, \fa \(\ve>0\),
\begin{equation}\label{e.width.ineq.2}
\la_{\ve}(\Om)\leq \tilde{\la}_{\ve}(\Om).
\end{equation}

The following Proposition \ref{p.intersection} is the Allen-Cahn analogue of \cite{CL}*{Proposition 2.1}.

\begin{pro}\label{p.intersection}
Let \(N\) be a closed \Rm manifold and \(\Om\subset N\) be a good set (as defined in Section \ref{s.2.2}). \sps \(f:N\ra [1/3,\infty)\) is a Morse \fn so that in the interval $[1/3,2/3]$, $f$ has no critical value which is a non-global local maxima or minima;
\[\min_N f=1/3;\quad \max_N f>1;\quad \Om\subset\subset f^{-1}\left([1/3,2/3)\right); \quad 1 \text{ is a regular value of }f.\]
We set 
$$\tilde{\Om}=f^{-1}\left([1/3,1]\right).$$
Then there exist \(\ve^*,\et^*>0\), depending on \(\Om,\;\tilde{\Om}\) and \(f\big|_{\tilde{\Om}}\), such that \fa \(0<\ve\leq \ve^*\) the following condition is satisfied. For every \(\ze\in \sB\), \tes \(t^0\in [0,1]\) \st 
\begin{equation}
\Ee{\ze(t^0)}\geq \tilde{\la}_{\ve}(\Om) \quad \text { and }\quad \Ee{\ze(t^0),\Om}\geq \et^*.
\end{equation}	
\end{pro}
\begin{rmk}\label{r.intersection}
The constants \(\ve^*\) and \(\et^*\) in the above Proposition \ref{p.intersection} depend on the ambient \Rm metric restricted to \(\tilde{\Om}\). (By our hypothesis, \(\del\tilde{\Om}\) is smooth.) Let us fix a \Rm \mt \(g_0\) on \(\tilde{\Om}\). If \(g'\) is an arbitrary \Rm \mt on \(N\), from the proofs of Proposition \ref{p.intersection}, Lemmas \ref{l.deformation.a} and \ref{l.deformation.b}, it follows that there exists \(\varrho>0\), depending on \(g_0\) and \(g'\big|_{\tilde{\Om}}\), \st the following holds. One can choose \(\ve^*\) and \(\et^*\) in Proposition \ref{p.intersection} in such a way that the proposition holds for all \Rm metrics \(g''\) on \(N\) satisfying
\begin{equation}
\nm{g'\big|_{\tilde{\Om}}-g''\big|_{\tilde{\Om}}}_{C^2(\tilde{\Om},g_0)}<\varrho.
\end{equation} 
\end{rmk}
\begin{proof}[Proof of Proposition \ref{p.intersection}]
The proof will be presented in four parts.

\textbf{Part 1.} Let
\begin{equation}\label{tau}
\tau=\frac{\si}{2}\cH^n(\del\Om).
\end{equation}
We set \(\et=\ta\) in Lemma \ref{l.deformation.a} and choose \(\ve^*_1>0\) and 
\begin{equation}\label{e.tau_1}
0<\ta_1< \cH^{n+1}(\Om)
\end{equation}
so that Lemma \ref{l.deformation.a} holds for \(\et=\ta\), \(\ve_1=\ve_1^*\) and \(\tilde{\et}=3\ta_1\). Next, we set \(\et_1=\ta_1\) in Lemma \ref{l.deformation.b} and choose \(\ve^*_2>0\) and 
\begin{equation}\label{tau2}
0<\ta_2\leq\frac{\si}{12}\cH^n(\del\Om)
\end{equation}
so that Lemma \ref{l.deformation.b} holds for \(\et_1=\ta_1\), \(\ve_2=\ve^*_2\) and \(\et_2=\ta_2\). Let us define \(\et^*=\ta_2\). We also define \(\ve^*\) to be a positive real number so that the following conditions are satisfied.
\begin{itemize}
	\item \(\ve^*\leq \min\{\ve^*_1,\ve^*_2\}\).
	\item Let
\begin{equation}\label{e.w.eps}
w_{\ve}= w_{1,\ve}(-2\sqrt{\ve}),
\end{equation}
where \(w_{1,\ve}\) is as defined in \eqref{e.def.w.1.eps}. Then, \fa \(0<\ve\leq \ve^*\),
\begin{equation}\label{e.En.eps}
\Ee{w_{\ve}}\leq 2\si \cH^n(\del\Om)+\ta_2.
\end{equation}
(For this item one needs to use \eqref{e.energy.w.1.eps}.)
\item For all \(0<\ve\leq \ve^*\), using the notation of \eqref{Om_r},
\begin{equation}\label{e.vol.eps}
\cH^{n+1}\left(\Om \setminus \Om_{-4\sqrt{\ve}}\right)\leq \frac{\ta_1}{2}.
\end{equation}
\item For all \(0<\ve\leq \ve^*\),
\begin{equation}\label{e.eps.good}
\frac{1}{2\si}\la_{\ve}(\Om)>\frac{7}{2}\cH^n(\del\Om).
\end{equation}
(For this item one needs \eqref{e.width.ineq.1} and the hypothesis \eqref{good.set} that \(\Om\) is a good set.)
\end{itemize}
\bigskip

We will show that Proposition \ref{p.intersection} holds for the above choices of \(\ve^*\) and \(\et^*\). Let us assume by contradiction that \te \(\al\in (0,\ve^*]\) and \(h\in \sB\) \st 
\begin{equation}\label{e.contr}
\text{for }t\in [0,1],\text{ if }\Ea{h(t)}\geq \tilde{\la}_{\al}(\Om)\;\text{  then  }\; \Ea{h(t),\Om}<\et^*=\ta_2.
\end{equation}
Without loss of generality, we can assume that 
\begin{equation}\label{h(t).bounded}
|h(t)|\leq 1\quad \forall\; t\in [0,1].
\end{equation}
Indeed, if 
\[\hat{h}(t)=\min\left\{1,\max\{-1,h(t)\}\right\},\]
then \fa \(t\in [0,1]\), \(|\hat{h}(t)|\leq 1\) and \(e_{\ve}(\hat{h}(t))\leq e_{\ve}(h(t))\). \tf 
\[\Ea{\hat{h}(t)}\geq \tilde{\la}_{\al}(\Om)\implies \Ea{h(t)}\geq \tilde{\la}_{\al}(\Om)\implies \et^*>\Ea{h(t),\Om}\geq \Ea{\hat{h}(t),\Om}.\]

To prove Proposition \ref{p.intersection}, we will show that the existence of such \(h\in \sB\) and \(\al\in (0,\ve^*]\) imply there exists \(\ga\in \sB\) satisfying \(\sup\limits_{t\in [0,1]}\Ea{\ga(t)}< \tilde{\la}_{\al}(\Om)\).
\medskip

\textbf{Part 2.} Let \(h\) and \(\al\) be as defined above in \eqref{e.contr} and \eqref{h(t).bounded}.

\begin{lem}\label{l.cont}
There exist \(0<a<b<1\) \st the following conditions are satisfied.
\begin{itemize}
	\item \(\Ea{h(a),\Om}=\ta_2=\Ea{h(b),\Om}\).
	\item \(\nm{1-h(a)}_{L^1(\Om)}\leq \ta_1\) and \(\nm{1+h(b)}_{L^1(\Om)}\leq \ta_1\).
	\item \(\Ea{h(t),\Om}\geq\ta_2\) \fa \(t\in [a,b]\).
\end{itemize}
\end{lem}
\begin{proof}
Let
\begin{align}
& S_1=\left\{t\in [0,1]:\Ea{h(t),\Om}\leq \ta_2\text{ and }\nm{1-h(t)}_{L^1(\Om)}\leq \ta_1 \right\};\\
& S_2=\left\{t\in [0,1]:\Ea{h(t),\Om}\leq \ta_2\text{ and }\nm{1+h(t)}_{L^1(\Om)}\leq \ta_1 \right\}.
\end{align}
By \eqref{e.tau_1}, 
\begin{equation}\label{e.disj}
\nexists\; t\in [0,1] \text{ \st } \max\left\{\nm{1-h(t)}_{L^1(\Om)}, \nm{1+h(t)}_{L^1(\Om)}\right\}\leq \ta_1.
\end{equation}
\eqref{e.disj}, together with the choices of \(\ta_1\), \(\ta_2\) and Lemma \ref{l.deformation.b}, implies that
\begin{equation}\label{e.disjoint}
S_1\cap S_2=\emptyset;\quad S_1\cup S_2=\left\{t\in [0,1]:\Ea{h(t),\Om}\leq \ta_2\right\}.
\end{equation}
\(S_1\) and \(S_2\) are closed subsets of \([0,1]\). Since \(h\in \sB\), \(0\in S_1\) and \(1\in S_2\). Let
\begin{equation}
a=\max S_1,\quad b=\min \left(S_2\cap [a,1]\right).
\end{equation}
\eqref{e.disjoint} implies that \(\nm{1+h(a)}_{L^1(\Om)}>\ta_1\). \sps \(\Ea{h(a),\Om}<\ta_2\). By continuity, \tes \(a'>a\) \st 
\[\Ea{h(a'),\Om}<\ta_2 \quad \text{ and } \quad \nm{1+h(a')}_{L^1(\Om)}>\ta_1,\]
 which implies (by \eqref{e.disjoint}) \(a'\in S_1\). This contradicts the definition of \(a\); hence \(\Ea{h(a),\Om}=\ta_2\). A similar argument shows that \(\Ea{h(b),\Om}=\ta_2\) as well. \sps there exists \(t'\in (a,b)\) \st \(\Ea{h(t'),\Om}<\ta_2\). Then by \eqref{e.disjoint}, \(t'\in S_1\cup S_2\), which contradicts the definitions of \(a\) and \(b\). This finishes the proof of the lemma.
\end{proof}

\textbf{Part 3.} By \eqref{e.contr} and Lemma \ref{l.cont}, \tes \(\de>0\) \st
\begin{equation}
\sup_{t\in [a,b]}\Ea{h(t)}\leq \tilde{\la}_{\al}(\Om)-\de.
\end{equation}
By Proposition \ref{p.approx.by.nested}, \tes a nested map \(\tilde{h}:[0,1]\ra \HN\) \st \(\htl(0)\geq h(a)\), \(\htl(1)\leq h(b)\), \(\md{\htl(t)}\leq 1\) \fa \(t\in [0,1]\) and
\begin{equation}\label{e.energy.htl}
\sup_{t\in [0,1]}\Ea{\htl(t)}\leq \tilde{\la}_{\al}(\Om)-\frac{\de}{2}.
\end{equation}

We recall from \eqref{e.w.eps} that $w_{\ve}= w_{1,\ve}(-2\sqrt{\ve})$; hence using the notation of \eqref{Om_r},
\begin{equation}\label{e.w.eps.1.-1}
w_{\ve}\equiv\begin{cases}
-1 & \text{ on } \Om_{-4\sqrt{\ve}};\\
1 & \text{ on } (N\setminus \Om).
\end{cases}
\end{equation}
\begin{lem}\label{l.map.T}
Let \(T:\HN\ra \HN\) be defined by
\begin{equation}
T(u)=\min\{-\we,\max\{\we,u\}\}.
\end{equation}
If \(\md{u}\leq 1\), then denoting \(\hat{u}=T(u)\), we have \(|\hat{u}|\leq 1\);
\begin{align}
& \Ea{\hat{u}}\leq \Ea{\we}+\Ea{u,\{-\we\geq u \geq \we\}}\leq \Ea{\we}+\Ea{u,\Om};\label{e.u'.1}\\
& \nm{1-\hat{u}}_{L^1(\Om)}\leq \nm{1-u}_{L^1(\Om)}+2\cH^{n+1}\left(\Om\setminus\Om_{-4\sqrt{\al}}\right);\label{e.u'.2}\\
& \nm{1+\hat{u}}_{L^1(\Om)}\leq \nm{1+u}_{L^1(\Om)}+2\cH^{n+1}\left(\Om\setminus\Om_{-4\sqrt{\al}}\right).\label{e.u'.3}
\end{align}
\end{lem}
\begin{proof}
\begin{equation}\label{e.u'.cases}
\hat{u}=\begin{cases}
u &\text{on } \{-\we\geq u \geq \we\};\\
\pm \we & \text{otherwise. }
\end{cases}
\end{equation}
\tf $|\hat{u}|\leq 1$. Moreover,
\begin{equation}\label{e.sub.Om}
\{-\we\geq u \geq \we\}\subset\{\we\leq 0\}\subset \Om\;\text{ (by \eqref{e.w.eps.1.-1})}.
\end{equation}
Combining \eqref{e.u'.cases} and \eqref{e.sub.Om}, one gets \eqref{e.u'.1}. It follows from \eqref{e.w.eps.1.-1} and \eqref{e.u'.cases} that 
\begin{equation}\label{e.u'=u}
\hat{u}=u\text{ on }\Om_{-4\sqrt{\al}}.
\end{equation}
Moreover,
\begin{equation}\label{e.1.pm.we}
0\leq 1\pm \we \leq 2.
\end{equation}
Equations \eqref{e.u'.2} and \eqref{e.u'.3} both follow from \eqref{e.u'.cases}, \eqref{e.u'=u} and \eqref{e.1.pm.we}.
\end{proof}
Let us define 
\begin{equation}\label{e.ze}
\ell=\min\{h(a),-h(b)\};
\end{equation}
so
\begin{equation}
-\ell=\max\{-h(a),h(b)\}.
\end{equation}
Moreover,
\begin{equation}\label{e.1-zeta}
0\leq 1-\ell\leq(1-h(a))+(1+h(b)).
\end{equation}
\begin{lem}\label{l.htl.t*}
Let $\htl$ be as in \eqref{e.energy.htl}. There exists \(t^*\in [0,1]\) \st
\begin{equation}\label{e.htl.t*}
\Eaa{\htl(t^*),\{\we\leq \htt \leq -\we\}^c\cup \{-\ell\leq \htt\leq \ell\}^c}\leq \Ea{\we}+2\ta_2.
\end{equation}
Here, for \(S\subset N\), \(S^c=(N\setminus S)\).
\end{lem}
\begin{proof}
Let \(h':[0,1]\ra \HN\) be defined by
\begin{equation}\label{e.h'(t)}
h'(t)=\min\{\ell,\max\{-\ell,\htl(t)\}\}=\begin{cases}
\htl(t) &\text{on }\{\ell\geq\htl(t)\geq -\ell\};\\
\pm\ell &\text{otherwise}.
\end{cases}
\end{equation}
Since
\[\ell\leq h(a)\leq \htl(0)\leq \max\{-\ell,\htl(0)\}\quad\text{ and }\quad -\ell\geq h(b)\geq \htl(1),\]
we have
\begin{equation}\label{e.h'.0.1}
h'(0)=\ell\quad\text{ and }\quad h'(1)=\min\{\ell,-\ell\}.
\end{equation}

Let \(h'':[0,1]\ra \HN\) be defined by \(h''(t)=T(h'(t))\), where \(T\) is as in Lemma \ref{l.map.T}. By Lemma \ref{l.map.T}, \eqref{e.h'.0.1}, \eqref{e.En.eps}  and Lemma \ref{l.cont},
\begin{align}
&\Ea{h''(0)}\leq \Ea{\we}+\Ea{\ell,\Om}\leq 2\si\cH^n(\del\Om)+3\ta_2;\\
& \Ea{h''(1)}\leq \Ea{\we}+\Ea{\ell,\Om}\leq 2\si\cH^n(\del\Om)+3\ta_2.
\end{align}
Moreover, by Lemma \ref{l.map.T}, \eqref{e.1-zeta}, Lemma \ref{l.cont} and \eqref{e.vol.eps},
\begin{align}
&\nm{1-h''(0)}_{L^1(\Om)}\leq \nm{1-\ell}_{L^1(\Om)}+2\cH^{n+1}\left(\Om\setminus\Om_{-4\sqrt{\al}}\right)\leq 3\ta_1;\\
&\nm{1+h''(1)}_{L^1(\Om)}\leq \nm{1-\ell}_{L^1(\Om)}+2\cH^{n+1}\left(\Om\setminus\Om_{-4\sqrt{\al}}\right)\leq 3\ta_1.
\end{align}
\tf by our choices of \(\ta_1\), \(\ta_2\), \(\ta\) and Lemma \ref{l.deformation.a}, \tes a \cts map \(\be_0:[0,1]\ra \HN\) \st \(\be_0(0)=h''(0)\), \(\be_0(1)\big|_{\Om}\equiv 1\) and
\begin{equation}\label{e.be_0}
\sup_{t\in [0,1]}\Ea{\be_0(t)}\leq 4\si\cH^n(\del\Om)+3\ta_2+\ta\leq5\si\cH^n(\del\Om)\;(\text{by }\eqref{tau}\text{ and }\eqref{tau2}).
\end{equation}
Similarly, \tes a \cts map \(\be_1:[0,1]\ra \HN\) \st \(\be_1(0)=h''(1)\), \(\be_1(1)\big|_{\Om}\equiv -1\) and
\begin{equation}\label{e.be_1}
\sup_{t\in [0,1]}\Ea{\be_1(t)}\leq 4\si\cH^n(\del\Om)+3\ta_2+\ta\leq5\si\cH^n(\del\Om).
\end{equation}
Let us define \(\be:[0,1]\ra \HN\) by
\begin{equation}
\be(t)=\begin{cases}
\be_0(1-3t) &\text{if }0\leq t\leq 1/3;\\
h''(3t-1) &\text{if }1/3\leq t\leq 2/3;\\
\be_1(3t-2) &\text{if }2/3\leq t\leq 1.\\
\end{cases}
\end{equation}
Since \(\be\in \sB\), \tes \(t^{\bullet}\in [0,1]\) \st 
\[\Ea{\be(t^{\bullet})}\geq \tilde{\la}_{\al}(\Om)\geq \la_{\al}(\Om)>7\si\cH^n(\del\Om)\;(\text{using \eqref{la.tld.ep}, \eqref{e.width.ineq.2} and \eqref{e.eps.good}}).\]
\tf by \eqref{e.be_0} and \eqref{e.be_1}, \tes \(t^*\in [0,1]\) \st
\begin{equation}\label{e.t*}
\Ea{h''(t^*)}\geq \tilde{\la}_{\al}(\Om).
\end{equation}
However, by Lemma  \ref{l.map.T}, \eqref{e.h'(t)} and Lemma \ref{l.cont}, for all \(t\in [0,1]\),
\begin{align}
\Ea{h''(t)}&\leq \Ea{\we}+\Eaa{h'(t),\{\we\leq h'(t)\leq -\we\}}\\
&\leq \Ea{\we}+\Eaa{\htl(t),\{\we\leq \htl(t)\leq -\we\}\cap \{-\ell\leq \htl(t)\leq \ell\}}+\Ea{\ell, \Om}\\
&\leq \Ea{\we}+\Eaa{\htl(t),\{\we\leq \htl(t)\leq -\we\}\cap \{-\ell\leq \htl(t)\leq \ell\}}+ 2\ta_2.\label{loc1}
\end{align}
Further, by \eqref{e.energy.htl},
\begin{equation}\label{loc2}
\Ea{\htl(t^*)}<\tilde{\la}_{\al}(\Om).
\end{equation}
Combining \eqref{e.t*}, \eqref{loc1} and \eqref{loc2}, one obtains \eqref{e.htl.t*}.
\end{proof}

\textbf{Part 4.} 
For \(r \in \bbr\), let \(r^{+}=\max\{r,0\}\); \(r^{-}=\min\{r,0\}\). The maps \(\Ph,\Ps,\Tht :H^1(N)\times H^1(N)\times H^1(N) \ra H^1(N)\) are defined as follows \cite{d}*{Equation (3.70)}.
\begin{align}
&\Ph(u_0,u_1,w)=\min\{\max\{u_0,-w\},\max\{u_1,w\}\}; \nonumber\\
&\Ps(u_0,u_1,w)=\max\{\min\{u_0,w\},\min\{u_1,-w\}\}; \label{def of phi, psi}\\
&\Tht(u_0,u_1,w)=\Ph(u_0,u_1,w)^{+}+\Ps(u_0,u_1,w)^{-}.\nonumber
\end{align}
\begin{lem}\label{l.tht}
Let \(u_0,u_1,w\in \HN\) \st \(|u_0|,|u_1|,|w|\leq 1\); \(\ph=\Ph(u_0,u_1,w)\), \(\ps=\Ps(u_0,u_1,w)\), \(\tht=\Tht(u_0,u_1,w)\).
\begin{itemize}
	\item[(i)] If \(w(x)=1\), then \(\tht(x)=u_0(x)\); if \(w(x)=-1\), then \(\tht(x)=u_1(x)\).
	\item[(ii)] If \(u_0(x)=u_1(x)\), then \(\tht(x)=u_0(x)=u_1(x)\).
	\item[(iii)] For all \(x\in N\), either \(\tht(x)=\ph(x)\) or \(\tht(x)=\ps(x)\); hence \(\tht(x)\in \{u_0(x),u_1(x),w(x),-w(x)\}\).
	\item[(iv)] For \(\ve>0\) and \(S\subset N\),
	\begin{align}
	\Ee{\tht,S}\leq \Ee{w,S}&+\Eee{u_0,S\cap\big(\{u_0>-w\}\cup\{u_0<w\}\big)}\\
	&+\Eee{u_1,S\cap\big(\{u_1>w\}\cup\{u_1<-w\}\big)}.
	\end{align}
	\item[(v)] If \(v_0,v_1\in \HN\) \st \(v_0\geq u_0,u_1\geq v_1\), then \(v_0\geq \tht\geq v_1\).
\end{itemize}
\end{lem}
\begin{proof}
For the proofs of items (i) -- (iii), we refer to \cite{d}*{Proof of Proposition 3.12}. (iv) follows from (iii) and the definitions of \(\Ph\) and \(\Ps\).
To prove item (v), we note the following. If \(r_1,r_1',r_2,r_2'\in \bbr\) \st \(r_1\geq r_2\) and \(r_1'\geq r_2'\), then \(\max\{r_1,r_1'\}\geq \max\{r_2,r_2'\}\) and \(\min\{r_1,r_1'\}\geq \min\{r_2,r_2'\}\). In particular, if we set \(r_1'=r_2'=s\), then \(\max\{r_1,s\}\geq \max\{r_2,s\} \) and \(\min\{r_1,s\}\geq \min\{r_2,s\}.\) \hn \(v_0\geq u_0,u_1\geq v_1\) implies
\begin{align}
&\Ph(v_0,v_0,w)\geq \Ph (u_0,u_1,w)\geq \Ph(v_1,v_1,w),\\
&\Ps(v_0,v_0,w)\geq \Ps (u_0,u_1,w)\geq \Ps(v_1,v_1,w).
\end{align}
\tf using item (ii), we obtain
\begin{equation}
v_0\geq \Tht(u_0,u_1,w)\geq v_1.
\end{equation}
\end{proof}
Let $t^*$ be as in Lemma \ref{l.htl.t*} and \(\ell\) be as defined in \eqref{e.ze}. We define
\begin{align}
&\ell_0=\max\{\htt,\ell\};\quad \ell_1=\min \{\htt,-\ell\};\\
& h^*_0=\Tht(\htt,\ell_0,\we);\quad h^*_1=\Tht(\htt,\ell_1,\we).
\end{align}
Using the fact that 
\begin{equation}
\ell_0=\begin{cases}
\ell &\text{on } \{\htt\leq \ell\}\\
\htt &\text{on } \{\htt>\ell\},
\end{cases}
\end{equation}
and Lemma \ref{l.tht} (ii), (iv), we obtain
\begin{align}
\Ea{h^*_0}
&=\Eaa{\Tht(\htt,\ell,\we),\{\htt\leq \ell\}}+\Eaa{\htt,\{\htt>\ell\}}\\
&\leq \Ea{\we}+\Eab{\ell,\{\ell>\we\}\cup \{\ell<-\we\}}\\
&\hspace{11ex}+\Eaa{\htt,\{\htt>\ell\}\cup\{\htt>-\we\}\cup\{\htt<\we\}}\\
&\leq 2\Ea{\we}+4\ta_2.\label{En.h*0}
\end{align}
In the last step we have used Lemma \ref{l.cont}, Lemma \ref{l.htl.t*} and the fact that 
\[\{\ell>\we\}\cup \{\ell<-\we\}\subset \{\we <1\}\subset \Om\;(\text{by } \eqref{e.w.eps.1.-1}).\]
By a similar argument,
\begin{equation}\label{En.h*1}
\Ea{h^*_1}\leq 2\Ea{\we}+4\ta_2.
\end{equation}
By Lemma \ref{l.tht} (i) and \eqref{e.w.eps.1.-1}, \(h^*_0=\ell_0\) on \(\Om_{-4\sqrt{\al}}\). Further, \(1\geq\ell_0\geq\ell\) and by Lemma \ref{l.tht} (v), \(\md{h^*_0}\leq 1\). Hence,
\begin{align}
\nm{1-h_0^*}_{L^1(\Om)}&\leq \nm{1-\ell_0}_{L^1(\Om_{-4\sqrt{\al}})}+2\cH^{n+1}\left(\Om\setminus\Om_{-4\sqrt{\al}}\right)\\
&\leq \nm{1-\ell}_{L^1(\Om)}+\ta_1\;(\text{by }\eqref{e.vol.eps})\\
&\leq 3\ta_1 \;(\text{by Lemma } \ref{l.cont}).\label{L^1.h^*.0}
\end{align}
Similarly, one can show that 
\begin{equation}\label{L^1.h^*.1}
\nm{1+h_1^*}_{L^1(\Om)}\leq 3\ta_1.
\end{equation}
Using Lemma \ref{l.tht} (v) and the definitions of \(\htl\) and \(\ell\), we obtain
\begin{equation}\label{chain.inequ}
\htl(0)\geq h^*_0\geq \htt;\quad \htt\geq h^*_1\geq\htl(1) .
\end{equation}

By Lemma \ref{l.minimization.prob}, \tes \(\htl(0)\geq h^{\bullet}_0\geq h^*_0\) \st
\begin{align}
\Ea{h^{\bullet}_0}&=\inf \{\Ea{u}:\htl(0)\geq u \geq h^*_0\}\\
&\leq \Ea{h^*_0}\\
&\leq 2\Ea{\we}+4\ta_2\;(\text{by }\eqref{En.h*0})\\
&\leq 4\si\cH^n(\del\Om)+6\ta_2\;(\text{by }\eqref{e.En.eps}).\label{e.En.h^bl.0}
\end{align}
Moreover, \eqref{L^1.h^*.0} and \(1\geq h^{\bullet}_0\geq h^*_0\) imply that
\begin{equation}\label{e.L^1.h^bl.0}
\nm{1-h_0^{\bullet}}_{L^1(\Om)}\leq \nm{1-h_0^*}_{L^1(\Om)}\leq 3\ta_1.
\end{equation}
Similarly, by Lemma \ref{l.minimization.prob}, \tes \(h^*_1\geq h^{\bullet}_1\geq \htl(1)\) \st
\begin{align}
\Ea{h^{\bullet}_1}&=\inf \{\Ea{u}:h^*_1\geq u \geq \htl(1)\}\\
&\leq \Ea{h^*_1}\\
&\leq 2\Ea{\we}+4\ta_2\;(\text{by }\eqref{En.h*1})\\
&\leq 4\si\cH^n(\del\Om)+6\ta_2\;(\text{by }\eqref{e.En.eps}).\label{e.En.h^bl.1}
\end{align}
Moreover, \eqref{L^1.h^*.1} and \(h^*_1\geq h^{\bullet}_1\geq -1\) imply that
\begin{equation}\label{e.L^1.h^bl.1}
\nm{1+h_1^{\bullet}}_{L^1(\Om)}\leq \nm{1+h_1^*}_{L^1(\Om)}\leq 3\ta_1.
\end{equation}

Setting $u=\htl$ and $v=h^{\bullet}_1$ in Lemma \ref{l.trancation} (a), we conclude that \tes a nested map \(\ga_1:[0,1]\ra \HN\) \st \(\ga_1(0)=\htl(0)\), \(\ga_1(1)=h^{\bullet}_1\) (\eqref{chain.inequ} implies that $\htl(0)\geq h^{\bullet}_1$)  and
\begin{equation}
\sup_{t\in [0,1]}\Ea{\ga_1(t)}\leq \sup_{t\in [0,1]}\Ea{\htl(t)}\leq \tilde{\la}_{\al}(\Om)-\frac{\de}{2}\;(\text{by }\eqref{e.energy.htl}).
\end{equation}
Next, setting $u=\ga_1$ and $v=h^{\bullet}_0$ in Lemma \ref{l.trancation} (b), we obtain another nested map \(\ga_2:[0,1]\ra \HN\) \st \(\ga_2(0)=h^{\bullet}_0\), \(\ga_2(1)=h^{\bullet}_1\) (\eqref{chain.inequ} implies that $h^{\bullet}_0\geq h^{\bullet}_1$) and
\begin{equation}\label{ga.2}
\sup_{t\in [0,1]}\Ea{\ga_2(t)}\leq \sup_{t\in [0,1]}\Ea{\ga_1(t)}\leq \tilde{\la}_{\al}(\Om)-\frac{\de}{2}.
\end{equation}

By the definitions of \(\ta_1\), \(\ta\) and Lemma \ref{l.deformation.a}, \eqref{e.En.h^bl.0} and \eqref{e.L^1.h^bl.0} imply that \tes \(\tilde{\ga}_0:[0,1]\ra \HN\) \st \(\tilde{\ga}_0(0)=h^{\bullet}_0\), \(\tilde{\ga}_0(1)\big|_{\Om}\equiv 1\) and
\begin{equation}\label{a1}
\sup_{t\in [0,1]}\Ea{\tilde{\ga}_0(t)}\leq 6\si\cH^n(\del\Om)+6\ta_2+\ta\leq 7\si\cH^n(\del\Om)\; (\text{by }\eqref{tau} \text{ and }\eqref{tau2}).
\end{equation}
Similarly \eqref{e.En.h^bl.1} and \eqref{e.L^1.h^bl.1} imply that \tes \(\tilde{\ga}_1:[0,1]\ra \HN\) \st \(\tilde{\ga}_1(0)=h^{\bullet}_1\), \(\tilde{\ga}_1(1)\big|_{\Om}\equiv -1\) and
\begin{equation}\label{a2}
\sup_{t\in [0,1]}\Ea{\tilde{\ga}_1(t)}\leq 6\si\cH^n(\del\Om)+6\ta_2+\ta\leq 7\si\cH^n(\del\Om)\; (\text{by }\eqref{tau} \text{ and }\eqref{tau2}).
\end{equation}
Let \(\ga:[0,1]\ra \HN\) be defined by
\begin{equation}
\ga(t)=\begin{cases}
\tilde{\ga}_0(1-3t) &\text{if }0\leq t\leq 1/3;\\
\ga_2(3t-1) &\text{if }1/3\leq t\leq 2/3;\\
\tilde{\ga}_1(3t-2) &\text{if }2/3\leq t\leq 1.\\
\end{cases}
\end{equation}
Then \(\ga\in \sB\); \eqref{ga.2}, \eqref{a1}, \eqref{a2}, \eqref{e.eps.good} and \eqref{e.width.ineq.2} imply that
\begin{equation}
\sup_{t\in [0,1]}\Ea{\ga(t)}< \tilde{\la}_{\al}(\Om),
\end{equation}
which contradicts the definition of \(\tilde{\la}_{\al}(\Om)\) (\eqref{la.tld.ep}). This finishes the proof of Proposition \ref{p.intersection}.
\end{proof}

\begin{thm}\label{t.second.thm}
Let \((N^{n+1},g)\), \(n+1\geq 3\), be a closed \Rm manifold and \(\Om\subset N\) be a good set. \sps \(\ve^*\) and \(\et^*\) are as in Proposition \ref{p.intersection} and Remark \ref{r.intersection} (where we set \(g'=g\) in Remark \ref{r.intersection}). Then for all \(0<\ve\leq \ve^*\), \tes \(\vartheta_{\ve}:N\ra (-1,1)\) satisfying \(AC_{\ve}(\vartheta_{\ve})=0\), \(\textup{Ind}(\vartheta_{\ve})\leq 1\), \(E_{\ve}(\vartheta_{\ve})=\tilde{\la}_{\ve}(\Om)\) and \(E_{\ve}(\vartheta_{\ve},\Om)\geq \et^*\).
\end{thm}
\begin{proof}
To prove this theorem, we apply Theorem \ref{t.min.max} to the functional \(E_{\ve}:\HN\ra \bbr\), \(0<\ve\leq \ve^*\). In Theorem \ref{t.min.max}, we set 
\begin{equation}
B_0=\left\{u\in \HN:u\big|_{\Om}\equiv 1\right\};\quad B_1=\left\{u\in \HN:u\big|_{\Om}\equiv -1\right\};
\end{equation}
\(\sF=\sB\) so that \(c=\tilde{\la}_{\ve}(\Om)\) and
\begin{equation}\label{e.def.L}
L= \{u\in \HN:\Ee{u}\geq \tilde{\la}_{\ve}(\Om)\}\cap\left\{u\in \HN:\Ee{u,\Om}\geq \et^*\right\}.
\end{equation}
Since \(u\big|_{\Om}\equiv 1\) or \(u\big|_{\Om}\equiv -1\) imply that \(\Ee{u,\Om}=0\), the condition (a1) of Theorem \ref{t.min.max} is satisfied. By Proposition \ref{p.intersection}, the condition (a2) is also satisfied. It follows from \eqref{e.def.L} that (a3) is satisfied as well. Following \cite{G}*{Section 4}, if \(\{\tilde{h}_i\}_{i=1}^{\infty}\) is an arbitrary minimizing sequence for \(E_{\ve}\) in \(\sB\), we define \(\{h_i\}_{i=1}^{\infty}\subset \sB\) by 
\[h_i(t)=\min\{1,\max\{-1,\tilde{h}_i(t)\}\}.\]
Then
\begin{equation}\label{pm.1}
-1\leq h_i(t)\leq 1\quad\forall\;i\in\bbn,\;t\in[0,1]
\end{equation}
and \(E_{\ve}(h_i)\leq E_{\ve}(\tilde{h}_i)\). Hence \(\{h_i\}\) is again a minimizing sequence. By \cite{G}*{Proposition 4.4}, \(E_{\ve}\) satisfies the Palais-Smale condition along \(\{h_i\}\). \tf by Theorem \ref{t.min.max}, part (a), \tes \(\vartheta_{\ve}\in \cK\left(\{h_i\}\right)\) \st \[AC_{\ve}(\vartheta_{\ve})=0,\quad \Ee{\vartheta_{\ve}}=\tilde{\la}_{\ve}(\Om),\quad \Ee{\vartheta_{\ve},\Om}\geq \et^*.\]
Moreover, by \eqref{pm.1}, \(\md{\vartheta_{\ve}}\leq 1\); hence by the strong maximum principle \(\md{\vartheta_{\ve}}< 1\). In addition, if the ambient \mt \(g\in \widetilde{\cM}\) (where \(\widetilde{\cM}\) is as defined in Theorem \ref{t.finite}) and \(\ve^{-1} \notin \text{Spec}(-\De_g)\), then \(\cZ_{\ve,g}\) is finite. In that case, the condition (b1) of Theorem \ref{t.min.max}, part (b) is satisfied and one can ensure that \(\vartheta_{\ve}\) satisfies \(\text{Ind}(\vartheta_{\ve})\leq 1\). 

To get the Morse index upper bound for arbitrary \mt \(g\) and \(\ve\in (0,\ve^*]\), we use an approximation argument. Since, by Theorem \ref{t.finite}, \(\widetilde{\cM}\) is a generic subset of \(\cM\), it is possible to choose \(\{g_i\}_{i=1}^{\infty}\subset \widetilde{\cM}\) \st \(g_i\) converges to \(g\) smoothly. Let \(\{\ep_i\}_{i=1}^{\infty}\) be a sequence in \((0,\ve^*]\) \st \(\ep_i^{-1}\notin\text{Spec}(-\De_{g_i})\) and \(\ep_i\ra \ve\). Since the width \(\mathbb{W}(\Om)\) depends continuously on the ambient metric \cite{IMN}*{Lemma 2.1}, \(\Om\) is a good set \wrt \(g_i\) if \(i\) is sufficiently large. \tf by the above discussion, Theorem \ref{t.finite}, Theorem \ref{t.min.max}, Proposition \ref{p.intersection} and Remark \ref{r.intersection} imply that for all \(i\) sufficiently large, \tes \(\vartheta^i:N\ra (-1,1)\) \st
\begin{equation}\label{appr}
AC_{\ep_i,g_i}(\vartheta^i)=0;\quad \text{Ind}_{g_i}(\vartheta^i)\leq 1;\quad E_{\ep_i,g_i}\left(\vartheta^i\right)=\tilde{\la}_{\ep_i,g_i}(\Om);\quad E_{\ep_i,g_i}\left(\vartheta^i,\Om\right)\geq \et^*.
\end{equation}
In this equation, the subscript \(g_i\) indicates that these quantities are computed \wrt the \mt \(g_i\). Since \(\md{\vartheta^i}<1\), by the elliptic regularity and the Arzela-Ascoli theorem, \tes \(\vartheta^{\infty}:N\ra [-1,1]\) \st up to a subsequence \(\vartheta^i\) converges to \(\vartheta^{\infty}\) in \(C^{2}(N)\). Using \eqref{appr} and the fact that the min-max quantity \(\tilde{\la}_{\ve}(\Om)\) depends continuously on the ambient metric \cite{GG2}*{Lemma 5.4}, we obtain
\begin{equation}
AC_{\ve,g}(\vartheta^{\infty})=0;\quad \text{Ind}_{g}(\vartheta^{\infty})\leq 1;\quad E_{\ve,g}\left(\vartheta^{\infty}\right)=\tilde{\la}_{\ve,g}(\Om);\quad E_{\ve,g}\left(\vartheta^{\infty},\Om\right)\geq \et^*.
\end{equation}
Furthermore, by the strong maximum principle, \(\md{\vartheta^{\infty}}<1\). This finishes the proof of Theorem \ref{t.second.thm}.
\end{proof}
\begin{proof}[Proof of Theorem \ref{t:main.thm}]
As proved in \cite{CL}*{Section 8.1}, \(\vol(M)<\infty\) implies that \tes a sequence \(\{U_i\}_{i=1}^{\infty}\), where each \(U_i\subset M\) is a bounded open set \w smooth boundary, \st \(U_{i}\subset U_{i+1}\) for all \(i\in \bbn\) and \(\lim\limits_{i\ra\infty}\cH^n(\del U_i)=0.\) As a consequence, \tes \(i_0\in\bbn\) \st \(U_{i_0}\) is a good set. For simplicity, let us denote \(U_{i_0}\) by \(U\).

The following proposition was proved in \cite{Mon}*{Section 12.2}.
\begin{pro}\cite{Mon}*{Section 12.2}\label{p.Morse.extn}
Let \(M'\) be a complete \Rm \mf and \(f':M'\ra[a_0,\infty)\) be a proper Morse function. \sps \(a_1\) is a regular value of \(f'\) and define
\[\cR=\{x\in M':f'(x)\leq a_1\}.\]
Then \tes a closed \Rm \mf \(N'\) and a Morse \fn \(f'':N'\ra [a_0,\infty)\) \st \(N'\) contains an isometric copy of \(\cR\), \(f''\) coincides with \(f'\) on \(\cR\) and \(f''>a_1\) on \(N'\setminus\cR\).
\end{pro}

Coming back to the proof of Theorem \ref{t:main.thm}, we choose a proper Morse \fn \(f_1:M\ra [0,\infty)\) with $\min\limits_M f_1=0$. Let \(t_1\) be a regular value of \(f_1\) so that \(U\subset\subset f_1^{-1}\left([0,t_1/2)\right)\). Furthermore, by suitably modifying $f_1$, we can assume that in the interval $[0,t_1/2]$, $f_1$ has no critical value which is a non-global local maxima or minima. Let
\[f_2=\frac{1}{3}+\frac{2f_1}{3t_1}.\]
Then \(f_2:M\ra [1/3,\infty)\), \(U\subset\subset f_2^{-1}\left([1/3,2/3)\right)\) and we set \(\tilde{U}=f_2^{-1}([1/3,1])\). 

We choose an increasing sequence \(\{s_i\}_{i=1}^{\infty}\) \st \(s_1\geq 1\), each \(s_i\) is a regular value of $f_2$ and \(s_i\ra \infty\) as \(i\ra \infty\). Let 
\[Q_i=\{x\in M:f_2(x)< s_i\},\quad \bar{Q}_i=\{x\in M:f_2(x)\leq s_i\}. \]
By Proposition \ref{p.Morse.extn}, \tes a \seq \(\{N_i\}_{i=1}^{\infty}\) of closed \Rm \mfs \st \(N_i\) contains an isometric copy of \(\bar{Q}_i\). Moreover, for every \(i\), \tes a Morse \fn \(\tilde{f}_i:N_i\ra [1/3,\infty)\) \st \(\tilde{f}_i\) coincides with \(f_2\) on \(\bar{Q}_i\) and \(\tilde{f}_i>s_i\) on \(N_i\setminus \bar{Q}_i\). In particular, \(N_i\) contains isometric copies of \(U\) and \(\tilde{U}\);  suppose \(\cU_i\) denotes the isometric copy of \(U\) in \(N_i\). Setting \(N=N_i\) and \(f=\tilde{f}_i\) in Proposition \ref{p.intersection} and Remark \ref{r.intersection} and using Theorem \ref{t.second.thm}, we obtain \(\ve_0,\;\et_0>0\), which depend only on \(U\), \(\tilde{U}\), the ambient metric on \(M\) restricted to \(\tilde{U}\) and \(f_2\big|_{\tilde{U}}\) \st the following holds.\footnote{In particular, \(\ve_0\) and $\et_0$ do not depend on $i$.} For all \(0<\ve\leq \ve_0\), \tes \(\vt_{\ve,i}:N_i\ra (-1,1)\) satisfying
\begin{equation}
AC_{\ve}(\vt_{\ve,i})=0;\quad \text{Ind}(\vt_{\ve,i})\leq 1;\quad \Ee{\vt_{\ve,i}}=\tilde{\la}_{\ve}\left(\cU_i\right);\quad \Ee{\vt_{\ve,i},\cU_i}\geq \et_0.
\end{equation}

Let \(b\in [1/3,\infty)\) such that
\begin{equation}
\left\{x\in M:d(x,U)\leq 2\sqrt{\ve_0}\right\}\subset f_2^{-1}\left([1/3,b]\right).
\end{equation}
One can modify \(f_2\) on \(f_2^{-1}\left([1/3,b]\right)\) and define another proper Morse \fn \(f_3:M\ra[1/3,\infty)\) so that 
\begin{equation}\label{U.contained}
\left\{x\in M:d(x,U)\leq 2\sqrt{\ve_0}\right\}\subset f_3^{-1}\left([1/3,b]\right)
\end{equation}
and in the interval $[1/3,b]$, $f_3$ has no critical value which is a non-global local maxima or minima. For \(1/3\leq t\leq b\), let \(d^t:M\ra \bbr\) be defined by
\begin{equation}
d^t(x)=\begin{cases}
-d(x,f^{-1}_3(t)) & \text{ if } f_3(x)\leq t;\\
d(x,f^{-1}_3(t)) & \text{ if } f_3(x)\geq t.
\end{cases}
\end{equation}
\eqref{U.contained} and the fact that $b$ is not a local maximum value of $f_3$ imply
\begin{equation}\label{subset}
U\subset \left\{d^{b}\leq-2\sqrt{\ve_0} \right\}.
\end{equation}
We choose \(i_1\in \bbn\) so that
\begin{equation}\label{intQ}
\left\{d^{b}\leq2\sqrt{\ve_0} \right\}\subset Q_{i_1}.
\end{equation}
For $0<\ve\leq \ve_0$, let \(\ze_{\ve}:[0,b] \ra H^1(Q_{i_1})\) be defined by 
\begin{equation}
\ze_{\ve}(t)=\begin{cases}
\hat{q}_{\ve}\circ d^t & \text{if } \frac{1}{3}\leq t \leq b \\
1-3t(1-\ze_{\ve}(1/3)) & \text{if } 0 \leq t \leq \frac{1}{3},
\end{cases}
\end{equation}
where \(\qte\) is as defined in \eqref{def.qte}. $\ze_{\ve}$ is continuous (see the discussion after equation \eqref{w.2.eps}), \(\ze_{\ve}(0)\equiv 1\) and \eqref{subset} implies that \(\ze_{\ve}(b)\big|_U\equiv -1.\) Further, by \cite{G}*{Section 9},
\begin{equation}
\La_{\ve}:=\sup_{0\leq t\leq b}\Ee{\ze_{\ve}(t)}
\end{equation}
satisfies
\begin{equation}\label{b2}
\limsup_{\ve\ra 0^{+}}\La_{\ve}\leq 2\si\sup_{1/3\leq t\leq b} \cH^n\left(f_3^{-1}(t)\right).
\end{equation}
For each \(i\geq i_1\), $N_i$ contains an isometric copy of $Q_{i_1}$. Hence \eqref{intQ} implies that \(\ze_{\ve}\) canonically defines a \cts map \(\ze_{\ve,i}:\left[0,b\right]\ra H^1(N_i)\) ($\ze_{\ve,i}(t)\equiv 1$ on \(N_i\setminus Q_{i_1}\) \fa $t\in[0,b]$) so that \(\ze_{\ve}(0)\equiv 1\) and \(\ze_{\ve}(b)\big|_{\cU_i}\equiv -1.\) Thus
\begin{equation}\label{b1}
\tilde{\la}_{\ve}\left(\cU_i\right)\leq \La_{\ve}\quad \forall\; i\geq i_1.
\end{equation}

Restricting \(\vt_{\ve,i}\) to the isometric copy of \(Q_i\) contained in \(N_i\), one gets \(\fu_{\ve,i}:Q_i\ra (-1,1)\) \st
\begin{equation}\label{e.b1}
AC_{\ve}(\fu_{\ve,i})=0;\quad \text{Ind}(\fu_{\ve,i})\leq 1;\quad \Ee{\fu_{\ve,i},Q_i}\leq\tilde{\la}_{\ve}\left(\cU_i\right);\quad \Ee{\fu_{\ve,i},U}\geq \et_0.
\end{equation}
For each fixed \(j\in \bbn\), by the elliptic estimates, \(\nm{\fu_{\ve,i}}_{C^{2,\al}(Q_j)}\) is uniformly bounded for all \(i>j\) (since \(\md{\fu_{\ve,i}}<1\) for all \(i\in \bbn\)). Using a diagonal argument and the Arzela-Ascoli theorem, we conclude that \tes \(\fu_{\ve}:M\ra [-1,1]\) \st a subsequence \(\{\fu_{\ve,i_k}\}\) converges to \(\fu_{\ve}\) in \(C^2_{\text{loc}}(M)\). Hence, using \eqref{e.b1} and \eqref{b1},
	\begin{equation}
	AC_{\ve}(\fu_{\ve})=0;\quad \text{Ind}(\fu_{\ve})\leq 1;\quad \Ee{\fu_{\ve}}\leq \La_{\ve};\quad \Ee{\fu_{\ve},U}\geq \et_0.
	\end{equation}
By the strong maximum principle, \(\md{\fu_{\ve}}<1\). Further, as mentioned in \eqref{b2}, \(\limsup_{\ve\ra 0^{+}}\La_{\ve}<\infty\). This completes the proof of Theorem \ref{t:main.thm}.
\end{proof}

\begin{proof}[Proof of Theorem \ref{t.minimal.hyp} using Theorem \ref{t:main.thm}]
Let 
\[\Om_1\subset\dots\subset\Om_{i}\subset\Om_{i+1}\subset\dots\] 
be an exhaustion of \(M\) by bounded open subsets \w smooth boundaries and \(U\subset\subset \Om_1\). Using Theorem \ref{thm interface} and a diagonal argument, \te a \seq \(\{\ep_i\}_{i=1}^{\infty}\) converging to \(0\) and a stationary, integral varifold \(V_k\) in \(\Om_k\) (for each \(k\in\bbn\)) \st the following conditions are satisfied.
\begin{itemize}
\item
\begin{equation}\label{eq0}
V\left[\fu_{\ep_i}\big|_{\Om_k}\right]\ra V_k
\end{equation}
in the sense of varifolds.
\item \(\text{spt}(V_k)\) is a minimal \hy \w optimal regularity in $\Om_k$.
\item 
\begin{equation}\label{e2}
	\nm{V_k}(\text{Clos}(U))\geq\frac{1}{2\si}\liminf_{i\ra \infty}E_{\ep_i}\left(\fu_{\ep_i},U\right).
\end{equation}
\item
\begin{equation}\label{e3}
\nm{V_k}(\Om_{k})\leq\frac{1}{2\si}\limsup_{i\ra \infty}E_{\ep_i}\left(\fu_{\ep_i},\Om_{k}\right)\leq \frac{1}{2\si}\limsup_{i\ra \infty}E_{\ep_i}\left(\fu_{\ep_i}\right).
\end{equation}
\end{itemize}
By \eqref{eq0}, \(V_i\mres\Om_j=V_j\) if \(i>j\). \tf \tes a stationary, integral varifold \(V\) in \(M\) \st \(V\mres\Om_i=V_i\) and \(\text{spt}(V)\) is a minimal \hy \w optimal regularity. Further, \eqref{e1}, \eqref{e2} and \eqref{e3} imply
\[0<\nm{V}\left(\text{Clos}(U)\right)\leq \nm{V}(M)<\infty.\]
\end{proof}
\medskip

\bibliographystyle{amsalpha}
\bibliography{AC_noncpt}

\end{document}